\documentclass[11pt,letterpaper]{amsart}
\usepackage[margin=1.5in]{geometry}

\usepackage{mathrsfs}
\usepackage{amsfonts}
\usepackage{latexsym,epsfig}
\usepackage{amsmath, amscd, amsthm,amssymb}
\usepackage{thmtools}
\usepackage{mathabx, stmaryrd}
\usepackage{microtype}

\usepackage[pdftex]{color}
\usepackage{comment}
\usepackage{marginnote}
\usepackage[colorlinks,citecolor=cyan,linkcolor=magenta]{hyperref}
\usepackage[nameinlink]{cleveref}
\usepackage{enumitem}
\usepackage[mathscr]{euscript}
\usepackage[normalem]{ulem}

\usepackage{todonotes}

\newcommand{\RR}[0]{\mathbb{R}}

\newcommand{\A}[0]{\mathcal{A}}

\newcommand{\rank}[0]{\text{rank}}

\newcommand{\ZZ}[0]{\mathbb{Z}}
\newcommand{\NN}[0]{\mathbb{N}}

\newcommand{\sing}{\mathrm{sing}}

\newcommand{\diam}{\mathrm{diam}}
\newcommand{\Fill}{\operatorname{Fill}}
\newcommand{\s}{\mathbf{s}}

\newcommand{\wt}{\widetilde}
\newcommand{\mc}{\mathcal}

\newcommand{\ol}{\overline}

\newcommand{\mr}{\mathring}

\newcommand{\bdy}{\partial} 

\newcommand{\teich}{\mathcal T}

\newcommand{\tet}{\mathfrak{t}}

\newcommand{\HFK}{\operatorname{HFK}}
\newcommand{\HF}{\operatorname{HF}}
\newcommand{\Scirc}{{\mathring{S}}}

\newcommand{\fol}[0]{\mathcal{F}}

\newcommand{\tri}{\tau}
\newcommand{\psub}{\mr W}

\newcommand{\Fix}{\operatorname{Fix}}

\newcommand{\quotient}[2]{{\raisebox{0.2em}{$#1$}
           \!\!\left/\raisebox{-0.2em}{\!$#2$}\right.}}

\usepackage{hyperref}

\numberwithin{equation}{section}

\newtheorem{theorem}{Theorem}[section]
\newtheorem{lemma}[theorem]{Lemma}
\newtheorem{proposition}[theorem]{Proposition}
\newtheorem{corollary}[theorem]{Corollary}

\newtheorem{claim}[theorem]{Claim}

\theoremstyle{definition}
\newtheorem{question}[theorem]{Question}
\newtheorem{definition}[theorem]{Definition}
\newtheorem{remark}[theorem]{Remark}

\newtheorem{example}[theorem]{Example}
\newtheorem{construction}[theorem]{Construction}

\newcommand{\FF}{\mathcal{F}}

\newcommand{\T}{\mathcal{T}}

\newcommand{\C}{\mathcal{C}}

\newcommand{\Mod}{\mathrm{Mod}}

\newcommand{\define}[1]{\textbf{#1}}

\newcommand\tsim{\kern-.4em\sim}

\renewcommand{\setminus}{\smallsetminus}

\newcommand{\AC}{\mathcal{AC}}
\newcommand{\from}{\colon} 
\newcommand{\p}{{\bf p}}
\renewcommand{\r}{{\bf r}}
\renewcommand{\s}{{\bf s}}
\renewcommand{\int}{\mathrm{int}}

\renewcommand{\phi}{\varphi}
\renewcommand{\epsilon}{\varepsilon}

\DeclareMathOperator{\vol}{Vol}

\definecolor{mutedgreen}{rgb}{.1,.75,0.15}
\definecolor{darkgreen}{rgb}{0,0.6,0.4}
\definecolor{purple}{rgb}{0.4,0,0.6}

\begin{document}

\title[On fixed points of pseudo-Anosov maps]{On fixed points of pseudo-Anosov maps}

\author{Tarik Aougab}
\address{Department of Mathematics\\ 
Haverford College}
\email{\href{mailto:taougab@haverford.edu}{taougab@haverford.edu}}

\author{David Futer}
\address{Department of Mathematics\\ 
Temple University}
\email{\href{mailto:dfuter@temple.edu}{dfuter@temple.edu}}

\author{Samuel J. Taylor}
\address{Department of Mathematics\\ 
Temple University}
\email{\href{mailto:samuel.taylor@temple.edu}{samuel.taylor@temple.edu}}

\begin{abstract}
We give a formula to estimate the number of fixed points of a pseudo-Anosov homeomorphism of a surface. When the homeomorphism satisfies a mild property called \emph{strong irreducibility}, the log of the number of fixed points is coarsely equal to the Teichm\"uller translation length. We also discuss several applications, including an inequality relating the hyperbolic volume of a mapping torus to the rank of its Heegaard Floer homology.
\end{abstract}

\date{\today}
\maketitle

\section{Introduction}
Let $S$ be an orientable finite-type surface without boundary, and with $\chi(S) < 0$. Consider a pseudo-Anosov homeomorphism $f \colon S \to S$ (see \Cref{Sec:pA} for the definition).  In this paper, we address the question of how combinatorial and geometric data associated to $f$ determines its number of fixed points. 

We begin by stating a simplified version of our main theorem in the presence of a mild condition on $f$, before giving some additional context to the problem. This is followed by our general result and applications to a question of Wright and to invariants of knots and $3$--manifolds.

\subsection*{Main results}

Say a homeomorphism $f \colon S \to S$ is \define{strongly irreducible} if for every essential simple curve $\alpha$ on $S$, the curves $\alpha$ and $f(\alpha)$ intersect essentially. (The terminology originated in~\cite{Schleimer:StronglyIrreducible}.) For a surface of sufficient complexity, this means the displacement of $f$ on the curve graph of $S$ is at least $2$. Note that this condition depends only on the isotopy class of $f$. 
We denote the set of fixed points of $f$ by $\Fix(f)$ and the Teichm\"uller translation length of $f$ by $\ell_{\T}(f)$. To simplify notation, we also declare $\log(0) = 0$.

For strongly irreducible maps, our main theorem is the following:

\begin{theorem} \label{th:intro_main}
If $f \colon S \to S$ is a strongly irreducible pseudo-Anosov, then
\[
 \log \# \Fix(f) \asymp  \ell_{\T}(f).
\]
\end{theorem}

Since the usual definition of a pseudo-Anosov map (\Cref{Sec:pA}) assumes that the domain of $f$ can have punctures but no boundary, this is our assumption throughout the paper.
The notation $a \asymp b$ means that the two quantities are equal up to additive and multiplicative error depending only on $S$. See \Cref{Sec:Coarse} for more details.

\Cref{th:intro_main} fits into the following context. 
First, if we denote the stretch factor of $f$ by $\lambda$ 
and note that $\ell_{\T}(f) = \log \lambda$, then a theorem 
of Thurston \cite{Thurston:GeometryDynamics} says that 
$\# {\rm Fix}(f^n) \sim \lambda^n$. So the novelty of the 
theorem is that it counts fixed points of $f$ rather than 
fixed points of sufficiently large powers. Second,  
\Cref{th:intro_main} asserts an upper and a lower bound on 
$\# \Fix(f)$. The upper bound is not new; indeed, the stronger 
bound $\log \# \Fix(f) \leq \ell_\teich(f) + O(\log |\chi(S)|)$ 
follows by considering a bounded-size Markov partition of $f$. 
See \Cref{Rem:UpperBoundMarkov} for  details. So the real content of \Cref{th:intro_main} is the lower bound on $\# \Fix(f)$, which is entirely new. 

It is also classically known that a pseudo-Anosov homeomorphism $f$ minimizes the number of fixed points among all maps homotopic to $f$. 
(See Thurston~\cite{Thurston:GeometryDynamics}, as well as Birman and Kidwell~\cite{BirmanKidwell}.) Thus, if $f$ is a strongly irreducible pseudo-Anosov,  \Cref{th:intro_main}
produces a lower bound on $\#\Fix(f')$ for any $f'  \sim f$. We may thus write
$\#\Fix(\varphi) = \#\Fix(f)$ for the mapping class $\varphi = [f]$ of a pseudo-Anosov $f$.

Finally, we note that the conclusion of \Cref{th:intro_main} is false without the hypothesis that $f$ is strongly irreducible. Indeed, Los produces infinitely many pseudo-Anosov homeomorphisms on the same surface without nonsingular fixed points \cite{Los:Infinite}. The following proposition is a mild improvement using a top-level construction.


\begin{proposition}\label{prop:FPF}
For $S$ a closed surface of sufficiently large genus $g$, there are fixed-point-free pseudo-Anosov homeomorphisms $f_n \colon S \to S$ such that $\ell_{\T}(f_n) \to \infty$ as $n \to \infty$.
\end{proposition}

One could also produce examples where the $f_n$ have orientable foliations (as in Los) or have the property that $\ell_{\mc P}(f_n) \to \infty$, where $\mc P = \mc P(S)$ is the pants graph. 
The point is that given any pseudo-Anosov $f$ that sends a curve to a disjoint curve, one can compose $f$ with Dehn twists about that curve to increase Teichm\"uller translation length without increasing the number of fixed points. Also, one can produce examples without fixed points using an easy construction within a fibered cone.
See \Cref{Sec:Examples} for details.

Recall that every mapping class $\varphi$ acts on the complex of curves $\C(S)$, with stable translation length $\ell_{\C(S)}(\varphi)$. 
See \eqref{Eqn:StableLength} for the precise definition. 
\Cref{th:intro_main} has the following corollary:

\begin{corollary}\label{Cor:CurveGraphFixedPoints}
For any mapping class $\varphi$,
we have
\[
\ell_{\C(S)}(\varphi) \prec \log \# \Fix(\varphi).
\]
\end{corollary}

\begin{proof}
To prove the claimed coarse inequality,  it suffices to assume $\ell_{\C(S)}(\varphi) > 1$. This implies $\varphi = [f]$ is pseudo-Anosov, as reducible and periodic mapping classes have stable translation length $\ell_{\C(S)}(\varphi) = 0$. It also implies $f$ is strongly irreducible.
Since the systole map $\teich(S) \to \C(S)$ is coarsely Lipschitz \cite{MasurMinsky1}, the result follows from \Cref{th:intro_main}.
\end{proof}

When $f$ is not strongly irreducible, we still obtain a coarse formula for its number of fixed points, but this is more involved. The formula is in terms of 
certain subsurface distances between the invariant laminations $\lambda^+, \lambda^-$ of $f$ (see \Cref{Sec:Projection} for definitions). 

We say that a subsurface $Y \subset S$ is \define{overlapped by $f$}  if $Y$ and $f(Y)$ intersect essentially. In particular, the full surface $S$ is overlapped by every $f$.
We define $\mc D_f$ to be the set of (isotopy classes of) overlapped essential subsurfaces of $S$. 

\begin{theorem}\label{th:main_2}
Let $f \colon S \to S$ be a pseudo-Anosov homeomorphism. 
Then 
\[
\log \# \Fix(f) \asymp \ell_{\AC(S)}(f) +  \sum_{\langle f \rangle Y \subset \mc D_f \setminus \{ S \}} [d_Y(\lambda^+,\lambda^-)] + \sum_{\langle f \rangle A \subset  \mc D_f} [\log d_A(\lambda^+,\lambda^-)],
\]
where the first sum is over $f$-orbits of proper non-annular subsurfaces and the second is over $f$-orbits of annuli. 
\end{theorem}

See \Cref{Def:HierarchyLikeSum} for background on sums with cutoffs, such as the two sums appearing in the above statement.

Of course, when $f$ is strongly irreducible, every subsurface is overlapped. The way in which \Cref{th:main_2} reduces to \Cref{th:intro_main} uses a new combinatorial formula for Teichm\"uller space translation length 
that we prove in \Cref{Sec:DistanceFormulas}.

\subsection*{Fixed points and surface subgroups of mapping class groups}

Besides being of  intrinsic interest, the search for estimates on fixed points of pseudo-Anosov homeomorphisms is motivated by the following question
 of Alex Wright:

\begin{question}[Wright; cf Question 5 of \cite{KentLeininger:AtoroidalBundles}]
\label{Ques:Wright}
Does there exists a fixed-point-free homeomorphism $f \colon S \to S$ of a closed, hyperbolic surface, with the property that every essential simple closed curve $\alpha$ on $S$ fills with its image $f(\alpha)$?
\end{question}

Futer \cite{Futer:Periodic}, and independently Wright, observed that any map $f$ as in the question must belong to a pseudo-Anosov mapping class. By Thurston's theorem \cite{BirmanKidwell}, the pseudo-Anosov homeomorphism isotopic to $f$ must itself have no fixed points.
Hence, we can assume that $f$ is itself pseudo-Anosov.

Wright showed that if such a pseudo-Anosov $f \colon S \to S$ exists, then the map $S \hookrightarrow S \times S \setminus \Delta$ given by $x \mapsto (x,f(x))$ induces an embedding from $\pi_1(S)$ into $\pi_1(S \times S \setminus \Delta)$, a surface braid group, with purely pseudo--Anosov image. 
See \cite[Lemma 9]{KentLeininger:AtoroidalBundles} for a proof. Subsequently, Kent and Leininger showed that the image of $\pi_1(S)$ must be distorted in the surface braid group \cite[Lemma 13]{KentLeininger:AtoroidalBundles}.

\Cref{th:intro_main} and \Cref{Cor:CurveGraphFixedPoints} illustrate the tension between the hypotheses of  \Cref{Ques:Wright}.
Furthermore, one consequence of \Cref{th:intro_main} is that on a fixed surface $S$ there are at most finitely many conjugacy classes of pseudo-Anosov maps satisfying \Cref{Ques:Wright}. Indeed, any $f$ satisfying Wright's question would have to be strongly irreducible, hence \Cref{th:intro_main} provides an upper bound on the translation length $\ell_\teich(f)$. (The reduction to finitely many conjugacy classes also follows from a theorem of Bowditch~\cite{Bowditch:AtoroidalBundles}.) We remark that an \emph{effective} version of 
\Cref{th:intro_main} would imply an explicit upper bound on $\ell_\teich(f)$, 
hence an explicit bound on the number of tetrahedra in the veering triangulation of the mapping torus $M_f$ \cite[Theorem 4.1]{AgolTsang:Dynamics}.
This would make it 
theoretically feasible to find all candidate mapping classes for any given surface $S$.

In very recent work, Kent and Leininger constructed the first examples of purely pseudo-Anosov surface subgroups of mapping class groups  \cite{KentLeininger:AtoroidalBundles} using a related approach. However, Wright's \Cref{Ques:Wright} is still open. 

\subsection*{Connections between $3$--manifold invariants}

Every homeomorphism $f \from S \to S$ defines a \define{mapping torus}, namely the $3$--manifold
\[
M_f = \,\quotient{S \times [0,1]}{(x, 1) \sim (f(x), 0)}.
\]
It is well-known that $M_f$ depends only on the mapping class $\varphi = [f] \in \Mod(S)$, hence we may write $M_\varphi$. By a theorem of Thurston (see Otal \cite{Otal:Fibered}), $M_\varphi$ admits a hyperbolic metric if and only if $\varphi$ is a pseudo-Anosov class.

The most natural geometric invariant of $M_f$ is its hyperbolic volume, denoted $\vol(M_f)$.  There is now an extensive literature relating $\vol(M_f)$ to the dynamical properties of the map $f$. Notably, Brock \cite{Brock:fibered} gave a two-sided estimate on $\vol(M_f)$ in terms of the 
pants graph translation length $\ell_{\mc P}(f)$ and the Weil--Petersson translation length $\ell_{WP}(f)$. Kojima and McShane \cite{KojimaMcShane} gave an explicit upper bound on $\vol(M_f)$ in terms of the Teichm\"uller translation length $\ell_\teich(f)$:
\begin{equation}\label{Eqn:KojimaMcShane}
\vol(M_f) \leq 3\pi | \chi(S) | \cdot \ell_\teich(f).
\end{equation}
In very recent work, Liu \cite{Liu:EntropyVolume} proved an upper bound on $\ell_\teich(f)$ in terms of $\vol(M_f)$ and the length of the shortest closed geodesic in $M_f$, which is independent of the topology of $S$.

From a very different perspective, Heegaard Floer theory is a powerful package of homological invariants associated to closed oriented $3$--manifolds and knots in $3$--manifolds. To a closed oriented $3$--manifold $Y$, the theory associates a finitely generated abelian group $\widehat{\HF}(Y)$ that splits as a direct sum over classes in $H^2(Y)$. (There are also more refined invariants $\HF^+(Y)$, $\HF^-(Y)$, and $\HF^\infty(Y)$, which are beyond the scope of our discussion.) To a null-homologous knot $K \subset Y$, it associates a finitely generated abelian group $\widehat{\HFK}(Y, K)$ that splits as a direct sum over gradings valued in $\ZZ$. See Ozsv\'ath and Szab\'o \cite{OS1, OS2, OS:Knot} and Rasmussen \cite{Rasmussen:Thesis} for the original papers defining these invariants, and \cite{OS:Introduction} for an excellent survey.

While hyperbolic geometry and Heegaard Floer homology have both had immense applications to the study of $3$--manifolds (and especially to Dehn surgery problems), there are very few known connections between the two sets of invariants. Given that the two theories have reached maturity, Lin and Lipnowski asked

\begin{question}[Lin--Lipnowski \cite{LinLipnowski:SWLength}]\label{Quest:LL}
 For a hyperbolic three-manifold $Y$, is there any relationship between the topological
invariants arising from the hyperbolic geometry of $Y$ (e.g.\ the volume, injectivity radius,
lengths of geodesics, etc.) and the invariants arising from Floer homology?
\end{question}

One promising avenue for building a bridge between hyperbolic and Floer-theoretic invariants comes from recently discovered connections between the Heegaard Floer invariants of a mapping torus and the fixed points of its monodromy. The connection originated in the work of Cotton-Clay on symplectic Floer homology \cite{CottonClay:Symplectic}.
Very recently, Liu \cite{Liu:EntropyVolume} used the work of Cotton-Clay and the relationship between several different Floer theories to show that if $S$ is a closed surface surface of genus $g(S) \geq 3$ and $f \from S \to S$ is a pseudo-Anosov map, 
\begin{equation}\label{Eqn:Liu}
\#\Fix(f) \leq 2 \cdot \rank_{\ZZ} \, \widehat{\HF}(M_f).
\end{equation}
See \cite[Proposition 4.3]{Liu:EntropyVolume} and \cite[computation on page 29]{Liu:EntropyVolume}.

For knots in $3$--manifolds, the connection between Heegaard Floer invariants and fixed points is even more precise. 
Indeed, if $K \subset Y$ is a hyperbolic knot with Seifert surface $S$ of genus $g$, and $M = Y \setminus K$ fibers over $S^1$ with fiber $S$ and monodromy $f$, Ni \cite{Ni:FixedPointNote, Ni:FixedPointsHFK} and Ghiggini and Spano \cite{GhigginiSpano} showed that
\begin{equation}\label{Eqn:NiGS}
\#\Fix(f) + 1 =   \rank \, \widehat{\HFK}(Y, K; g(S)-1).
\end{equation}
This result also uses Cotton-Clay \cite{CottonClay:Symplectic} in an essential way.

Combining the above results with \Cref{th:intro_main} gives a partial answer to  \Cref{Quest:LL}.

\begin{corollary}\label{Cor:ClosedHF}
Let $S$ be a closed surface of genus $g(S) \geq 3$. Let $f \from S \to S$ be a strongly irreducible pseudo-Anosov. Then the mapping torus $M_f$ satisfies
\[
\vol(M_f) \prec \log \operatorname{rank}_{\ZZ} \widehat{\HF}(M_f),
\]
where the coarse constants depend only on the genus $g(S)$.
\end{corollary}

\begin{proof}
We compute as follows:
\begin{align*}
\vol(M_f) & \leq 3\pi | \chi(S) | \cdot \ell_\teich(f) \\
& \prec \log \#\Fix(f) \\
& \leq \log \rank_{\ZZ} \, \widehat{\HF}(M_f) + \log 2.
\end{align*}
The fist inequality is \eqref{Eqn:KojimaMcShane}, the second is \Cref{th:intro_main}, and the third is \eqref{Eqn:Liu}.
\end{proof}

\begin{corollary}\label{Cor:KnotHFK}
Let $Y$ be a closed orientable $3$--manifold. Let $K \subset Y$ be a knot whose complement $M = Y \setminus K$ is hyperbolic and fibered over $S^1$, with fiber a Seifert surface $S$ and strongly irreducible monodromy $f$. Then
\[
\vol(Y \setminus K) \prec \log \operatorname{rank} \, \widehat{\HFK}(Y, K),
\]
where the coarse constants depend only on the genus $g(S)$.
\end{corollary}

\begin{proof}
This is identical to the proof of \Cref{Cor:ClosedHF}, using \eqref{Eqn:NiGS} in place of \eqref{Eqn:Liu}.
\end{proof}

\begin{remark}[Weakening strong irreducibility]
By using \Cref{th:main_2} in place of \Cref{th:intro_main}, one can obtain many variations of \Cref{Cor:ClosedHF,Cor:KnotHFK}. For example, for both corollaries it suffices to assume that $f$ satisfies the weaker property that it overlaps all subsurfaces that are \emph{not} annuli or pants. In this case, \Cref{th:main_2} implies that 
\[
\log \# \Fix(f) \: \succ \:  \ell_{\AC(S)}(f) +  \sum_{\langle f \rangle Y} [d_Y(\lambda^+,\lambda^-) ]
\: \asymp \: \ell_{\mc P}(f).
\]
Here, the sum ranges over all $f$--orbits of non-annular subsurfaces, and $\ell_{\mc P}(f)$ is the translation length of $f$ in the pants graph. The coarse equality between the sum and $\ell_{\mc P}(f)$ is proved by 
essentially the same argument as \Cref{Prop:RafiInterpreted}.
Since $\ell_{\mc P}(f) \asymp \vol(M_f)$ by a result of Brock \cite{Brock:fibered}, the conclusion of the corollaries follows.
\end{remark}

\subsection*{Proof strategy and organization}
The central objects in the proof of \Cref{th:intro_main,th:main_2} are veering triangulations, singular flat metrics on surfaces, and distance formulas involving subsurface projections. All of these topics are reviewed in \Cref{Sec:background}. At a very top level, estimates on fixed points come from intersections between a veering edge and its image; intersection numbers can be estimated in terms of subsurface distances; and translation lengths in Teichm\"uller space can likewise be estimated in terms of subsurface distances.

We work in a singular flat metric on $S$ whose horizontal and vertical foliations are the stable/unstable foliations of the map $f$. In such a metric (given by a quadratic differential $q$), a special role is played by  the saddle connections  that span singularity-free rectangles with sides along the horizontal and vertical leaves. A key observation (recorded as the Fundamental Lemma in \Cref{Sec:FundamentalLemma}) is that for every singularity-free saddle connection $\sigma$, every intersection point in $\sigma \cap f(\sigma)$ gives rise to a fixed point of $f$. The correspondence between intersection points and fixed points is not quite one-to-one, but the degree of overcounting can be controlled.

Singularity-free saddle connections also arise as edges in the \define{veering triangulation}; see \Cref{maximal_R}. Letting $\mr S$ denote the complement of the singularities of $S$, this is an ideal triangulation of $\mr S \times \RR$ that enjoys very strong geometric and combinatorial properties. By a theorem of Minsky and Taylor \cite[Theorem 1.4]{MinskyTaylor:FiberedFaces}, veering edges form a totally geodesic subset of $\AC(\mr S)$. In \Cref{Sec:Tools}, we review and build on their work \cite{MinskyTaylor:FiberedFaces,MinskyTaylor:SubsurfaceDist} to show that veering triangulations are also nearly geodesic in every essential subsurface of $\mr S$. This allows us to use sections of the veering triangulation to estimate the distances between the stable and unstable laminations $\lambda^\pm$ in every essential surface (\Cref{sec:aas}).

As mentioned above, there are combinatorial estimates for distance in $\teich(S)$ and for intersection numbers between triangulations that package together a sum of distances in subsurfaces. We review these distance formulas (due to Rafi \cite{Rafi:Combinatorial} and Choi--Rafi \cite{choi2007comparison}) in \Cref{Sec:DistanceFormulas}, while adapting Rafi's formula to give an estimate for the translation length $\ell_\teich(S)$.

In \Cref{Sec:StronglyIrreducibleProof}, we assemble the above ingredients to prove \Cref{th:intro_main}. In \Cref{Sec:WithoutStronglyIrreducible}, we prove \Cref{th:main_2}, which requires all of the above ingredients in addition to some technical results from Vokes \cite{Vokes} and Kopreski \cite{Kopreski}. In \Cref{Sec:Examples}, we discuss examples and prove \Cref{prop:FPF}.

 \vspace{-1ex}
\subsection*{Acknowledgements}
We thank Braeden Reinoso for a beautiful talk that inspired this project.
We thank Matthew Durham for a helpful discussion about hierarchy-like sums (\Cref{Def:HierarchyLikeSum})
and Kasra Rafi for helpful discussions regarding the results in \Cref{Sec:DistanceFormulas}.
We thank Chi Cheuk Tsang for his comments on an early draft.

We thank Francesco Lin for several useful pointers on the connections between Heegaard Floer theory and fixed points. John Baldwin and Robert Lipshitz deserve particular gratitude for fielding many questions on about the same topic, and for clarifying  the proof of \eqref{Eqn:Liu}. 

During this project,  Futer was partially supported by NSF grant DMS--2405046. Taylor was partially supported by NSF grants DMS--2102018, DMS--2503113 and the Simons Foundation.


\section{Background} \label{Sec:background} 
Here we collect some basic background, mostly related to tools from surface topology and the mapping class group. The standard reference for most of this material is Masur and Minsky \cite{MasurMinsky2}.

\subsection{Coarse equalities, distance formulas, and cutoffs}\label{Sec:Coarse} 
Fix a surface $S$.
For quantities $a$ and $b$ and a constant $C\ge1$, we write $a \prec_C b$ to mean ``$a \leq C \cdot b + C$''. The equivalence relation  $a \asymp_C b$ means ``$a \prec_C b$ and $b \prec_C a$''.
As a common shorthand, we will write $a \prec b$ or $a \asymp b$ to mean that the implicit constant $C$ depends only on the surface $S$, whereas $a,b$ depend on other parameters. The relation $\prec$ is called \emph{coarse inequality}, and $\asymp$ is called \define{coarse equality}.

For a constant $K \geq 0$, the cutoff function $[ \cdot ]_K \from \RR \to \RR$ is defined so that 
\[
[a]_K =
\begin{cases}
a & a \geq K, \\
0 & a < K.
\end{cases}
\]
Cutoffs often appear in distance formulas, which are of the form
\[
A \asymp_C \sum_{Y \in \mathcal Y} [a_Y]_K
\]
where $a_Y$ is a nonnegative quantity associated to some (equivalence class of) essential subsurfaces $Y \subset S$ drawn from a collection $\mathcal Y$. A paradigmatic example is Masur and Minsky's distance formula in $\Mod (S)$
\cite[Theorem 6.12]{MasurMinsky2}.

We wish to remove the explicit mention of the constants $C$ and $K$ from distance formulas. This can be done as follows.

\begin{definition}\label{Def:HierarchyLikeSum}
Let $\mathcal Y$ be a collection of (equivalence classes of) essential subsurfaces of $S$.
The expression
\[
A \asymp \sum_{Y \in \mathcal Y} [a_Y]
\]
means that there is a constant $K_0 \geq 0$ such that for any $K \geq K_0$ there exists a constant $C>1$ (depending only $K$ and the surface $S$), so that 
\[
A \asymp_C \sum_{Y \in \mathcal Y} [a_Y]_K.
\]
If this holds, then the expression $A \asymp \sum [a_Y]$ is called a \define{hierarchy-like sum}.
\end{definition}

For example, the Masur--Minsky distance formula in $\Mod (S)$ \cite[Theorem 6.12]{MasurMinsky2} is a hierarchy-like sum. This is proved in Minsky~\cite[Section 9.4]{Minsky:ModelsBounds}; see \cite[Lemma 9.6]{Minsky:ModelsBounds} in particular. Similarly, Rafi's combinatorial distance formula in $\teich(S)$ \cite[Theorem 6.1]{Rafi:Combinatorial} is a hierarchy-like sum. See \cite[Remark 2.5]{Rafi:Hyperbolicity}, and compare \Cref{Prop:RafiInterpreted}.

The following elementary lemma underlies many later computations involving hierarchy-like sums.

\begin{lemma}
\label{Lem:HierarchyLikeSum}
Let $\mathcal Y$ be a collection of subsurfaces of $S$. Suppose $a_Y$ and $b_Y$ are quantities associated to the subsurfaces $Y \in \mathcal Y$, and $c \geq 1$ is a constant depending only on $S$ such that $a_Y \asymp_c b_Y$ for each $Y$. If $ \sum_{\mathcal Y}  [a_Y]$ is a hierarchy-like sum, then $\sum_{\mathcal Y} [b_Y]$ is also a hierarchy-like sum, and $\sum_{\mathcal Y}  [a_Y] \asymp \sum_{\mathcal Y} [b_Y]$.
\end{lemma}

\begin{proof}
 For each $Y$, we have that $a_{Y} \leq c \cdot b_{Y} + c$.  For any $K \geq c$, we will show that
 \begin{equation} \label{comparable}
 \sum_{Y} [a_Y]_{2cK} \leq (c+1) \sum_{Y} [b_Y]_{K}. 
 \end{equation}
 Now, suppose that $[a_Y]_{2cK}$ is a nonzero summand on the left side of \eqref{comparable}, meaning $a_{Y} \geq 2cK$. 
 Then
 \[ b_{Y} \geq \frac{a_{Y}- c}{c} \geq \frac{2cK -c}{c} = 2K-1 \geq K, \]
where the last inequality holds because $K \geq c \geq 1$. Therefore $b_Y = [b_Y]_K$ appears as a nonzero summand on the right side of \eqref{comparable}. 

To complete the proof, it remains to show that whenever $[b_Y]_K > 0$, we have $a_Y \leq (c+1) b_Y$. Since $b_Y \geq K \geq c$, we have
\[
a_Y \leq c \cdot b_Y + c \leq c \cdot b_Y + b_Y = (c+1) b_Y,
\]
as desired. This proves \eqref{comparable}.
The reverse inequality is proved in a totally analogous way by swapping the roles of $a_Y$ and $b_Y$. 
\end{proof}

\subsection{Arc and curve graphs}
Let $S$ be an orientable surface of finite type, potentially with non-empty boundary or punctures. 
An \define{arc} in $S$ is a proper embedding of an interval (possibly open, closed, or neither), where ``proper'' means that an endpoint of the interval is mapped into $\partial S$ and the preimage of a compact set is compact. Less formally, this latter condition means that each open end of the interval is mapped to an end of $S$. By an isotopy of an arc, we mean an isotopy through such proper embeddings.

The \define{arc and curve graph} $\AC(S)$ is the graph whose vertices are isotopy classes of essential simple closed curves and essential arcs in $S$. Here, \define{essential} means that the curve or arc is not isotopic into a small neighborhood of a point, a boundary component, or a puncture. 
Two vertices are joined by an edge in $\AC(S)$ if they have disjoint representatives.
If we follow the same construction with vertices restricted to be closed curves on $S$, we obtain the \define{curve graph} $\C(S)\subset \AC(S)$, and similarly restricting to arcs yields the \define{arc graph} $\A(S)\subset \AC(S)$.

When $S = A$ is an annulus, the definition of $\AC(A)$ is modified so that isotopies of arcs are required to fix endpoints pointwise, so that $\AC(A)$ becomes quasi-isometric to $\ZZ$; see Masur and Minsky~\cite{MasurMinsky2} for details. 
When $S$ is a pair of pants, $\C(S) = \emptyset$ and $\AC(S)$ is finite, hence pants are generally ignored in hierarchy-like sums. 
In all other cases, the inclusion $\C(S) \hookrightarrow \AC(S)$ is a $2$--quasi-isometry \cite[Lemma 2.2]{MasurMinsky2}.  We will primarily work in $\AC(S)$.

When equipped with the standard path metric, the graphs $\C(S)$ and $\AC(S)$ are $\delta$--hyperbolic. See Masur and Minsky \cite{MasurMinsky1} for the original proof. Aougab \cite{Aougab:UniformHyperbolicity}, Bowditch \cite{Bowditch:UniformHyperbolicity}, Clay--Rafi--Schleimer \cite{CRS:UniformHyperbolicity}, and Hensel--Przytycki--Webb \cite{HPW:SlimUnicorns} proved that the hyperbolicity constant $\delta$ is independent of $S$.

The \define{mapping class group} of $S$, denoted $\Mod (S)$, is the group of isotopy classes of orientation preserving homeomorphisms of $S$. Note that $\Mod(S)$ acts on $\AC(S)$ by graph automorphisms. For a mapping class $\varphi \in \Mod(S)$, we define the \define{stable translation length}
\begin{equation}\label{Eqn:StableLength}
\ell_S(\varphi) = \lim_{n \to \infty} \frac {d_{\AC(S)}(\alpha, \varphi^n(\alpha))}{n},
\end{equation}
where $\alpha \in \AC(S)$ is an arbitrary reference vertex. It follows from triangle inequalities that the limit exists and does not depend on $\alpha$.

\subsection{Subsurface projection}\label{Sec:Projection}
 A subsurface $Y \subset S$ is called \define{essential} if the inclusion map $Y \hookrightarrow S$ is $\pi_1$-injective, and $Y$ is neither a pair of pants nor a peripheral annulus.
We will often refer to $Y$ as simply a subsurface of $S$. 
 
There is a partially defined projection operator $\pi_{Y}$ from $\AC(S)$ to $\AC (Y)$, defined as follows (again, we refer to \cite{MasurMinsky2} for additional details).
Let $S_Y$ the the cover of $S$ corresponding to the subgroup $\pi_{1}(Y) < \pi_{1}(S)$. Fix an arbitrary complete hyperbolic metric on $S$, and consider the Gromov compactification $\overline S_Y$ of $S_Y$.
Then, given any essential simple closed curve or arc $\alpha$ on $S$, its unique geodesic representative admits a lift $\tilde{\alpha}$ to $\overline S_Y$. We define $\pi_{Y}(\alpha)$ to be the subset of (isotopy classes of) curves and arcs in this lift that are essential in $\overline S_Y$; this can be either a curve, an arc, a collection of pairwise disjoint curves or arcs, or empty in the event that $\alpha \cap Y = \emptyset$. As such, $\pi_{Y}$ is technically a map from $\AC(S)$ to the power set of $\AC (Y)$, but when nonempty the diameter of the image of $\alpha$ is always at most one. 

Given two curves or arcs $\alpha, \beta \in \AC(S)$, we define their $Y$-\define{subsurface distance} to be the diameter of the projection $\pi_{Y}(\alpha) \cup \pi_{Y}(\beta)$ in $\AC(Y)$:
\[
d_Y(\alpha, \beta) = \diam_{\AC(Y)} (\pi_{Y}(\alpha) \cup \pi_{Y}(\beta)).
\]
The definition extends immediately to other $1$--dimensional objects, such as multi-arcs and laminations whose leaves essentially intersect $\partial Y$.
 We adopt the convention that $\diam(\emptyset) = 0$. In particular, if $\pi_{Y}(\alpha) = \pi_{Y}(\beta) = \emptyset$, then $d_Y(\alpha, \beta) = 0$. 

The \define{bounded geodesic image theorem}, proved by Masur and Minsky \cite[Theorem 3.1]{MasurMinsky2}, will serve as a key tool in controlling projections of geodesics to subsurfaces: 

\begin{theorem} \label{Thm:BGIT} 
There is a constant $M$ depending only on $S$ such that if $Y \subset S$ is a proper essential subsurface and $g \subset \AC(S)$ 
is a geodesic that does not enter the $1$--neighborhood of $\bdy Y$ in $\AC(S)$, then 
\[ \mathrm{diam}(\pi_{Y}(g)) < M.\]
\end{theorem}

Webb~\cite{Webb:BGIT} gave a more combinatorial proof of Theorem \ref{Thm:BGIT}, showing that $M$ can be taken to depend only on the hyperbolicity constant of $\mathcal{AC}(S)$. Since there is a uniform hyperbolicity constant $\delta$ that does not depend on $S$ 
\cite{Aougab:UniformHyperbolicity, Bowditch:UniformHyperbolicity, CRS:UniformHyperbolicity, HPW:SlimUnicorns},
 it follows that $M$ can be taken to be a single universal constant. 

\subsection{Pseudo-Anosov homeomorphisms}\label{Sec:pA}
Let $S$ be a surface without boundary, possibly with punctures.
A homeomorphism $f \colon S \to S$ is called \define{pseudo-Anosov} if there is a pair of transverse singular measured foliations $\mc F^s, \mc F^u$ on $S$ and a constant $\lambda >1$, called its \define{stretch factor}, so that $f(\mc F^u) = \lambda \cdot \mc F^u$ and $f(\mc F^s) = \frac{1}{\lambda} \cdot \mc F^s$. The measured foliations $\mc F^{s/u}$ are called the \define{stable/unstable} foliations of $f$. 

Given the singular measured foliations $\mc F^s$ and $\mc F^u$, there is a unique singular flat structure $X$ on $S$ 
whose singularities are precisely those of $\mc F^{s/u}$, and whose whose vertical and horizontal foliations are $\mc F^s$ and $\mc F^u$, respectively. Since the measures on $\mc F^{s/u}$ are unique up to scaling, the singular flat structure $X$ is also unique up to horizontal and vertical rescaling.
In particular, our convention is that $f$ stretches along the horizontal leaves and contracts along the vertical leaves of $X$. More formally, $X$ can be described as a marked complex structure on $S$ together with a quadratic differential $q$ so that the image of $X$ under $f$ is equal to the image of $q$ under the time $\log(\lambda)$ Teichm\"uller flow. 

Throughout the paper, it will often be convenient to puncture $S$ at the singularities of $f$. The resulting surface is denoted $\mr S$ and there is an induced pseudo-Anosov homeomorphism $\mr f \colon \mr S \to \mr S$. 

\subsection{Veering triangulations}\label{Sec:Veering}

Veering triangulations were first defined and studied by Agol \cite{Agol:Veering}, with a later description by Gu\'eritaud \cite{Gueritaud:Veering} in terms of the singular flat structure on a surface. Our description follows Gu\'eritaud's definition, with notation and conventions following Minsky and Taylor \cite{MinskyTaylor:FiberedFaces, MinskyTaylor:SubsurfaceDist}. We refer to those sources for additional details.

We continue to use the notation of \Cref{Sec:pA}, where $X$ is a singular flat structure associated to the pseudo-Anosov map $f$.
An (immersed) \define{rectangle} in $X$ is the image of an embedded rectangle in the universal cover $\wt X$, with its induced singular flat structure. By a rectangle in $\wt X$ we mean the embedded image of a (necessarily singularity-free) Euclidean rectangle whose vertical sides map to segments in $\wt \fol^s$ and whose horizontal sides map to segments of $\wt \fol^u$. A \define{saddle connection} in $X$ or $\wt X$ is a straight segment between two singularities, i.e. a geodesic in the flat metric that meets the singularities exactly at its endpoints. Particularly important to this paper are the saddle connections that span rectangles; these will be called either \define{singularity-free saddle connections} or \define{veering edges}, depending on the perspective. In slightly more detail, these are the saddle connections that arise by taking a rectangle $\wt R$ in $\wt X$ that contains two singularities in its boundary and connecting them with the straight line segment to produce a saddle connection $\sigma$. The induced (sub)-rectangle of $\wt R$ whose opposite (singular) corners lie at the endpoints $\sigma$ is called the \define{spanning rectangle} $R_\sigma$ of $\sigma$. The same terminology applies in $X$ after applying the covering projection (although spanning rectangles in $X$  are only immersed; see \Cref{Lem:FundamentalLemma,Lem:LowDegreeAwayFromPocket}). 

Besides spanning rectangles,
we also define \define{maximal rectangles} to be rectangles that are not contained in a larger rectangle. In $\wt X$, these are the embedded rectangles that have exactly one singularity in the interior of each of its four sides. 



Every maximal (singularity-free) rectangle $R$ defines an oriented ideal tetrahedron $\tet \subset \mr S \times \RR$ with a map $\tet \to R$, as follows. The vertices of $\tet$ map to the four singularities in the boundary of $R$. The edges of $\tet$ map to the six saddle connections spanned by these four vertices. The orientation of $\tet$ is determined by the convention that the more-vertical edge (whose endpoints are on the horizontal edges of $R$) lies above the more-horizontal edge. See \Cref{maximal_R}.

\begin{figure}[htbp]
\begin{center}
\includegraphics{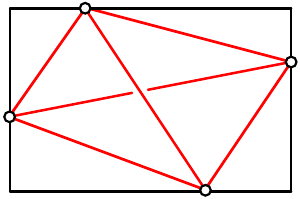}
\caption{A maximal rectangle and its associated ideal tetrahedron}
\label{maximal_R}
\end{center}
\end{figure}

Performing this construction for all maximal rectangles
gives a countable collection of tetrahedra whose ideal vertices map to singularities of $X$. If tetrahedra $\tet$ and $\tet'$ contain the same triple of saddle connections (equivalently, if maximal rectangles ${R}$ and ${R}'$ intersect along a sub-rectangle that meets three punctures), we glue $\tet$ to $\tet'$ along their shared face. By a theorem of Gu\'eritaud \cite{Gueritaud:Veering} (see also \cite[Theorem 2.1]{MinskyTaylor:FiberedFaces}), the resulting $3$--complex is an ideal triangulation $\tri_X$ of $\Scirc \times \RR$. Moreover, the maps of tetrahedra to their defining rectangles piece together to form a fibration $\Pi \from \Scirc \times \RR \to \Scirc$. We call $\tri_X$ the \define{veering triangulation associated to $X$}.


 The map $f \colon X \to X$ preserves the foliations $\FF^s, \FF^u$ and
permutes the maximal rectangles, inducing  a simplicial map $F \colon \tau_X \to \tau_X$. Accordingly, $\tau_X$ projects to an ideal triangulation of the punctured mapping torus $\mr M_f = (\mr S \times \RR) / F$.

Observe that a saddle connection of $X$ corresponds to an edge of $\tri_X$ if and only if it spans a singularity-free rectangle. In other words, edges of $\tri_X$ are in natural bijection with veering edges; hence the name.


A \define{section} of $\tau_X$ is simply a triangulation of $X$ by singularity-free saddle connections. We also think of a section as an ideal triangulation of $\wt X$.
The name comes from the fact that sections are in natural bijective correspondence with simplicial sections of the fibration $\Pi \from \Scirc \times \RR \to \Scirc$. 

Given a section $T$, an upward diagonal exchange is defined as follows: locate an edge $e$ that is wider (i.e. has greater mass with respect to $\mc F^s$) than the other edges in either of the two triangles $t_1,t_2$ containing $e$. Then there is a maximal rectangle $R$ containing $t_1,t_2$ such that $e$ joins the singularities in the vertical edges of $R$. The upward diagonal exchange replaces $e$ with the saddle connection that joins the horizontal edges of $R$ to give a new section $T'$. See 
\Cref{maximal_R}, where the reader can also deduce the $3$-dimensional interpretation of an upward diagonal exchange in $\tau_X$ (or they can consult   \cite[Section 3]{MinskyTaylor:FiberedFaces}). The reverse of an upward diagonal exchange is called a downward diagonal exchange.

According to Minsky and Taylor, any collection of noncrossing singularity-free saddle connections can be completed to a section \cite[Lemma 3.2]{MinskyTaylor:FiberedFaces} and any two such sections differ by a sequence of diagonal exchanges leaving the initial set of saddle connections untouched \cite[Proposition 3.3]{MinskyTaylor:FiberedFaces}. 

\begin{remark}\label{Rem:AboveBelow}
Thinking of a section $T$ as an ideal triangulation of $\mr S$ that simplicially embeds in the ideal triangulation $\tau_X$ of $\mr S \times \RR$, there is a natural sense in which each veering edge $\sigma$ either lies in $T$, below $T$, or above $T$. In terms of the triangulation $T$, this directly translates as follows: $\sigma$ lies below $T$ if whenever an edge $e$ of $T$ crosses $\sigma$ it crosses its spanning rectangle $R_\sigma$ from bottom to top (i.e. $e$ crosses the horizontal sides of $R_\sigma$). Equivalently, whenever $\sigma$ crosses an edge $e$ of $T$, it crosses $R_e$ from side to side (i.e. $\sigma$ crosses the vertical sides of $R_e$). The analogous statement can be made when $\sigma$ lies above $T$. Several variants of this ordering will be useful below.
\end{remark}

Conflating the distinction between foliations and laminations, we will refer to the vertical foliation $\fol^s$ as $\lambda^+$ and the horizontal foliation $\fol^u$ as $\lambda^-$ in the remainder of the paper. This matches the notation in the Introduction.

\section{The fundamental lemma}\label{Sec:FundamentalLemma}
This section describes a straightforward but crucial result that enables us to estimate the number of fixed points of $f$ by counting the intersections between a veering edge $\sigma$ and its image. This result is so fundamental to our estimates that we call it the \define{fundamental lemma}.

Here and throughout, we fix a singular flat structure $X$ on $S$ associated to $f$, whose vertical and horizontal foliations are denoted $\lambda^+$ and $\lambda^-$, respectively. 

Let $R = R_\sigma \subset S$ be an (immersed) spanning rectangle, as in \Cref{Sec:Veering}, with 
an embedded lift $\wt R  \subset \wt S$ and a projection map $\pi \from \wt R \to R$. The \define{degree} of $R$ is the maximal cardinality of a preimage $\pi^{-1}(z)$ for an interior point $z \in \int(R)$.


The following straightforward lemma relating fixed points to intersections is the primary way in which we find and count fixed points of a pseudo-Anosov map $f$.

\begin{lemma}[Fundamental lemma]\label{Lem:FundamentalLemma}
Let $\sigma \subset S$ be a 
singularity-free saddle connection, 
with immersed spanning rectangle $R = R_\sigma$ of degree $k$. Then the number of nonsingular fixed points of $f$ contained in $R$ is between 
$\frac{i(\sigma, f(\sigma))}{k}$ and $i(\sigma, f(\sigma))$.
\end{lemma}

Here and below, we emphasize that quantities such as $i(\sigma, f(\sigma))$ and $i(T, f(T))$ count interior, transverse intersection points contained in $\mr S$.

\begin{proof}
Let $\wt \sigma$ be a lift of $\sigma$ to $\wt S$, and let $\wt R \subset \wt S$ be the singularity-free rectangle spanned by $\wt \sigma$.
For every (nonsingular, interior) point of intersection $x_i \in \sigma \cap f(\sigma)$, we choose a lift $\wt f_i$ such that $\wt x_i \in \wt \sigma \cap \wt f_i(\wt \sigma)$ is a lift of $x_i$. Then 
the interiors of $\wt R$ and $\wt f_i(\wt R)$ have nonempty intersection.
We emphasize that distinct intersection points in $\sigma \cap f(\sigma)$ lift to distinct points in $\wt \sigma$ and lead to distinct lifts of $f$, and drop the subscripts going forward.

Since $f(\sigma)$ is more horizontal than $\sigma$ (see \Cref{Rem:AboveBelow}), it follows that $\wt f(\wt \sigma)$ crosses the singularity-free rectangle $\wt R$ from left to right. The endpoints of $\wt f(\wt \sigma)$ lie outside the closure of $\wt R$, because the intersection point $\wt \sigma \cap f(\wt \sigma)$ is in the interior of $\wt \sigma$. Similarly, the endpoints of $\wt \sigma$ lie strictly above and below the closure of $\wt f(
\wt R)$.

The map $\wt f$ acts affinely on $\wt R$, contracting the vertical direction and stretching the horizontal direction. By the contraction mapping theorem, there is a unique vertical coordinate in $\wt R$ that is fixed by $\wt f$. By the same theorem applied to $\wt f^{-1}$, there is a unique horizontal coordinate in $\wt f(\wt R)$ that is fixed by $\wt f^{-1}$. At the intersection of these coordinates lies the  unique fixed point of $\wt f$ in $\wt R$. This projects to a fixed point of $f$ in the interior of $R$. 

Since $R$ has degree $k$, and since distinct points of $\sigma \cap f(\sigma)$ define distinct lifts of $f$, the above construction produces at least $\frac{i(\sigma, f(\sigma))}{k}$ distinct fixed points in $\int(R)$. 

In the other direction, observe that every nonsingular fixed point of $f$ in $R$ must belong to $\int(R)$: for, if a point $x \in \bdy R$ is fixed, then the vertical or horizontal segment from $x$ to a singular corner of $R$ would also be fixed, contradicting the stretching behavior of $f$.
Thus every nonsingular fixed point  $x \in R$ lifts to a fixed point  of some lift $\wt f$ contained in $\int(\wt R)$. This lift  $\wt x$ must correspond to an intersection point $\wt \sigma \cap \wt f(\sigma)$ because $\wt R$ is a singularity-free rectangle. Hence the fixed point  $x \in R$ must be the projection of one of the $ i(\sigma, f(\sigma))$ fixed points in $\wt R$ constructed above.
\end{proof}



To apply the Fundamental Lemma for bounding fixed points from below, we will need an upper bound on the degrees of singularity-free rectangles. The next lemma characterizes the only case where high degree rectangles can occur.

\begin{lemma}\label{Lem:LowDegreeAwayFromPocket}
Let $\sigma$ be a singularity-free saddle connection whose immersed spanning rectangle $R$ has degree $k \geq 2$. Then there is a flat annulus $A$ in $S$ such that $\sigma$ joins singularities on either side of $A$. Furthermore, we have the following distance estimates:
\[
d_A(\lambda^+,\lambda^-) \geq 2k \quad \text{ and } \quad
\min \{ d_A(\sigma, \lambda^-), \, d_A(\sigma, \lambda^+)\} \geq k+1.
\]
\end{lemma}

\begin{figure}
\begin{tikzpicture}[scale=2]

\def\n{5}               
\def\x{0.23}             
\def\y{0.2}            
\def\k{4}               

\pgfmathsetmacro{\dxL}{\x - 0}     
\pgfmathsetmacro{\dyL}{0 - \y}     

\pgfmathsetmacro{\dxLp}{1 - (1 - \x)}   
\pgfmathsetmacro{\dyLp}{(1 - \y) - 1}   

\foreach \i in {0,...,4} {
  \begin{scope}[shift={(\i*\x, -\i*\y)}]
    \draw[thick] (0,0) rectangle (1,1);
    
    \ifnum\i=0
      \node at (0.3,1.2) {$\tilde{R}_0$};
    \fi
    \ifnum\i=4
      \node at (0.7,-.2) {$\tilde{R}_{k-1}$};
    \fi
  \end{scope}
}

\coordinate (p) at (0,\y);           
\coordinate (q) at (\x,0);           

\coordinate (p') at ({1 - \x}, 1);   
\coordinate (q') at (1, {1 - \y});   

\coordinate (Lstart) at (p);
\coordinate (LprimeStart) at (p');

\pgfmathsetmacro{\LendX}{0 + \k * (\x + \dxL)}
\pgfmathsetmacro{\LendY}{\y + \k * (-\y + \dyL)}

\pgfmathsetmacro{\LprimeEndX}{(1 - \x) + \k * (\x + \dxLp)}
\pgfmathsetmacro{\LprimeEndY}{1 + \k * (-\y + \dyLp)}

\coordinate (Lend) at (\LendX,\LendY);
\coordinate (LprimeEnd) at (\LprimeEndX,\LprimeEndY);

\draw[orange, thick, dashed] (Lstart) 
  -- node[pos=0.95, below right] {$L$} (Lend);

\draw[orange, thick, dashed] (LprimeStart) 
  -- node[pos=0.95, above right] {$L'$} (LprimeEnd);

\foreach \pt in {p, q, p', q'} {
  \fill[orange] (\pt) circle (0.6pt);
}

\node[orange] at (-0.25,1.1) {$\tilde{A}$};

\node at ({0.5*\x}, -0.12) {$x$};

\node at (1.1, {1 - 0.5*\y}) {$y$};

\node at (1.85, -0.65) {$\ddots$};

\end{tikzpicture}
\caption{Overlapping rectangles in the proof of \Cref{Lem:LowDegreeAwayFromPocket}.}
\label{Fig:OverlappingRectangles}
\end{figure}
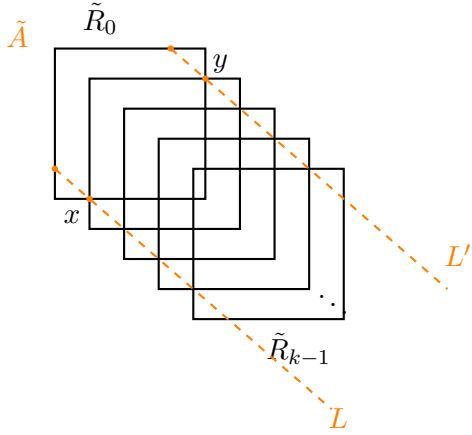

\begin{proof}
We begin by making some simplifying assumptions. Let $\wt R_0$ be a preimage of $R$ in $\widetilde S$. 
Assume without loss of generality that a lift $\wt \sigma$ spans $\wt R_0$ bottom-left to top-right. We rescale the metric $q_t$ (along the Teichm\"uller geodesic) so that $\wt R_0$ is a square, and assume without loss of generality that the sidelength is $1$. Via the developing map $\widetilde S \to \RR^2$, we identify $\wt R_0$ with the unit square of $\RR^2$.
Since the degree is $k$, there are deck transformation images $\wt R_1 = h_1(\wt R_0), \ldots, \wt R_{k-1} = h_{k-1}(\wt R_0)$ such that $\wt R_0 \cap \dots \cap \wt R_{k-1} \neq \emptyset$.
Observe that the orientation-preserving isometry $h_i$ must act on $\wt R_0$ by pure translation: 
 otherwise, the region between two innermost diagonals would descend to an immersed M\"obius band in $S$.
 Thus we may arrange $\wt R_0, \ldots, \wt R_{k-1}$ in order, top-left to bottom-right, with $h_i$ translating downward and to the right. 

Let $h = h_1$, and observe that any subsequent $h_i$ must be a power of $h$: otherwise there would be another image of $\wt R_0$ that is not accounted for above, and the degree would be higher than $k$. Thus we may set $h_i = h^i$ for each $i$.

Let $x,y$ be the horizontal and vertical translation lengths of $h$. Since $\wt R_0 \cap \wt R_{k-1} \neq \emptyset$, we have $\max(x,y) < \frac{1}{k-1}$. Let $\widetilde A$ be the region bounded by lines
\[
L \text{ through $(x,0)$ and $(0,y)$}  \quad \text{ and } \quad L' \text{ through $(1-x,1)$ and $(1,1-y)$};
\]
see \Cref{Fig:OverlappingRectangles}.
Then
\[
\widetilde A \subset \bigcup_{n \in \ZZ} h^n (\wt R_0).
\]
In particular, $\widetilde A$ has no singularities in its interior, hence it is a bi-infinite strip embedded in $\widetilde S$.
Since $\widetilde A$ is an $h$--invariant strip, it projects to a flat annulus $A \subset S$.

To compute distances in $A$, it suffices to count intersection points between $q_t$--geodesics. Let $\ell_-$ be a nonsingular leaf of $\lambda^-$ that runs through $\wt R_0 \cap \ldots \cap \wt R_{k-1}$. Then $\ell_-$ crosses each $\wt R_i$ left to right, hence it intersects the diagonals $\sigma_0, \sigma_1 = h(\sigma_0), \ldots, \sigma_{k-1} = h^{k-1}(\sigma_0)$. Downstairs in the annulus $A$, we have
\[
i(\ell_-, \sigma) \geq k \quad \text{hence} \quad d_A(\lambda^-, \sigma) \geq k+1.
\]
By the same argument, $d_A(\lambda^+, \sigma) \geq k+1$.

Finally, if we orient $\ell_+$ upward and $\ell_-$ (defined in the same manner) rightward, then $\sigma$ intersects $\ell_-$ and $\ell_+$ with opposite signs. Hence, $d_A(\lambda^+, \lambda^-) \ge i(\ell_+, \sigma) + i(\ell_-, \sigma) \ge 2k$. See e.g. \cite[Equation 2.5]{MasurMinsky2}.
\end{proof}

\begin{remark}
If the quadratic differential $q_t$ is chosen so that $R$ is a square, meaning $\sigma$ makes angle $\pi/4$ with both foliations, one can show that the modulus of $A$ satisfies
\[
\frac{k-2}{\sqrt{2}} \leq M(A, q_t) \leq \sqrt{2} k.
\]
Recall the modulus of a flat annulus is height over the circumference.
See also Leininger--Reid \cite[Proposition 3.1]{leininger2020pseudo} for a similar statement.
\end{remark}

\medskip

\section{Tools from veering triangulations}\label{Sec:Tools}
This section contains a number of useful technical results about subsurface projections and veering triangulations. Taken together, these results will enable us to use veering triangulations to estimate the sort of subsurface distances that appear in the statement of \Cref{th:main_2}.
Most of the lemmas of this section rely on results proved in Minsky and Taylor \cite{MinskyTaylor:FiberedFaces, MinskyTaylor:SubsurfaceDist}.

Throughout this section, we assume that $f \from S \to S$ is a pseudo-Anosov map and $q$ is a quadratic differential along the invariant axis of $f$. Let $\mr S$ denote the result of puncturing $S$ at the singularities of $q$.

\subsection{Compatibility and projections}\label{Sec:Compatibility}

We begin by reviewing two maps that are defined and studied in \cite{MinskyTaylor:FiberedFaces}. Our emphasis will be on subsurfaces of the punctured surface $\mr S$, although many constructions also work for subsurfaces of $S$. To streamline the discussion, let $X$ denote either $S$ or $\mr S$, together with the fixed flat structure $q$.

To begin, define $\A(q)$ to be the graph whose vertices are saddle connections of $q$ and whose edges correspond to saddle connections with nonintersecting interiors. Define $\mc A(\tau)$ to be the subgraph of $\A(q)$ spanned by the singularity-free saddle connections, i.e. the veering edges. By puncturing at the singularities of $q$, we have natural inclusions $\A(\tau) \subset \mc A(q) \subset \AC(\mr S)$.

Let 
$\s \from \AC(X) \to \A(q)$
 be the set-valued map that sends every 
essential curve or arc  $\alpha \subset X$
to the set of noncrossing saddle connections of $q$ that appear in some flat geodesic representative(s) of $\alpha$. (This geodesic representative is unique, except for the case where $\alpha$ is the core curve of a flat cylinder.)
See \cite[Section 4.3]{MinskyTaylor:FiberedFaces} for details. By \cite[Lemma 4.4]{MinskyTaylor:FiberedFaces}, $\s$ is a \define{coarse $1$--Lipschitz retraction}. That is, $\s$  sends sets of diameter $\leq 1$ to sets of diameter $\leq 1$, and restricts to the identity on its codomain.

Let $\r \colon \A(q) \to \A(\tau)$ be the set-valued rectangle hull map defined in \cite[Section 4.1]{MinskyTaylor:FiberedFaces}. Working in the universal cover, this map sends any saddle connection $\sigma$ to the set of veering edges whose endpoints are singularities occurring on the boundary of maximal rectangles whose diagonals lie along $\sigma$. 
By \cite[Lemma 4.1]{MinskyTaylor:FiberedFaces}, $\r$ is also a coarse $1$--Lipschitz retraction. Consequently, the composition
$\p = \r \circ \s \from \AC(X) \to \A(\tau)$  is also a coarse $1$--Lipschitz retraction. 
It follows that any two vertices in $\A(\tau)$ are joined by a geodesic of $\AC(\mr S)$ that lies in $\A(\tau)$ \cite[Theorem 1.4]{MinskyTaylor:FiberedFaces}.

\begin{figure}
\includegraphics{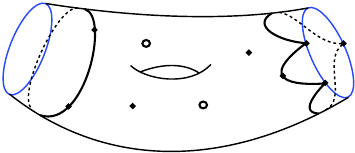}
\caption{
The subsurface $W \subset X$ is $q$--compatible and the image shows $W_q \subset \wt X_W$.
Open circles are punctures of $S$, while closed dots are singularities of $S$ that are punctured in $\mr S$. The solid black intervals are saddle connections whose concatenation forms $\bdy_q W$. 
This figure is a reproduction of \cite[Figure 2]{MinskyTaylor:SubsurfaceDist}.}
\label{Fig:YqCartoon}
\end{figure}

Let $W \subset X$ be an essential subsurface. If $\bdy W \neq \emptyset$, the \define{$q$--representative} of $\bdy W$ is $\bdy_q W = \s(\bdy W)$. 
In the case where $X = \mr S$, we also define the  \define{$\tau$--representative} of $\bdy W$ as 
$\bdy_\tau W = \p(\bdy W) = \r(\bdy_q W)$. There is a definition in the general case, but we will not need it here.

Next, we construct a closed set $W_q$, 
which can be thought of as the convex core of $W$ in the $q$--metric. See \cite[Section 2.4]{MinskyTaylor:SubsurfaceDist} for details. First suppose that $W$ is not an annulus. 
The set $W_q$ naturally lives in $\wt X_W$, which is the cover of $X$ associated to $W$. 
We lift the inclusion $\iota \from W \to X$ to an inclusion $\hat \iota \from W \to \wt X_W$,
and co-orient each component of $\bdy \hat \iota(W)$ outward, so that the co-orientation points to the annular components of the complement. If some component $\gamma$ of $\bdy \hat \iota(W)$ is a cylinder curve of $q$, as on the left side of \Cref{Fig:YqCartoon}, we isotope the flat geodesic representative inward with respect to the co-orientation, away from the ideal boundary of $\wt X_W \setminus \bdy_q \hat \iota(W)$. As a result, each flat geodesic representative $\gamma_q = \s(\gamma)$ of $\bdy \hat \iota(W)$ is chosen so that the outward side (pointed to by the co-orientation) is maximal with respect to inclusion.
Each outward side is either an annulus or a finite number of open disks; see 
\Cref{Fig:YqCartoon}. Let $W_q$ be the closed subset of $\wt X_W$ obtained by removing these (open) outward sides.

When $W$ is an annulus, $W_q$ is the maximal flat cylinder in $\wt X_W$, when it exists, and is otherwise the unique geodesic representative of its core. 

%

By  \cite[Lemma 2.6]{MinskyTaylor:FiberedFaces}, the subset $W_q \subset \wt X_W$ is $q$--convex. Furthermore, there is a deformation retraction $\wt X_W \to W_q$, which collapses each outward side to the flat geodesic $\gamma_q$. Thus $W_q$ is a $q$--convex homotopy representative of $W$. However, $W_q$ might fail to be a surface, and might even have empty interior. 

For $x\in \{q, \tau \}$ 
we say that $W$ is \define{$x$--compatible} if there is a component $\int_x W$ of $X \setminus \bdy_x W$ that is an (open) representative of the homotopy class of the subsurface $W$. In particular, if $W$ is $q$--compatible, then $\int_q W $ can be obtained by projecting 
 $\int(W_q) \subset \wt X_W$ down to $X$.
See \cite[Section 2.4]{MinskyTaylor:SubsurfaceDist}  for more details.

\begin{lemma}\label{Lem:CompatibleProj}
Let $W \subset X$ be a proper subsurface. If $d_W(\lambda^-,\lambda^+)\ge 4$, then $W$ is $q$--compatible and $\tau$--compatible.
\end{lemma}

\begin{proof}
See \cite[Proposition 2.4 and Theorem 2.5]{MinskyTaylor:SubsurfaceDist}. 
\end{proof}

Note that it follows from the definitions that $\tau$--compatibility implies $q$--compatibility, since `pinching' cannot be undone by applying $\r$.

\begin{lemma}\label{lem:stuck}
A subsurface $W \subset \mr S$ fails to be $x$--compatible if and only if $\bdy_x W$ meets $W$ essentially (meaning that $\pi_{W}(\bdy_x W) \neq \emptyset$). 
\end{lemma}

\begin{proof}
If $W$ is $x$--compatible, then every component of $\bdy_x W$ is disjoint from the interior $\int_x W$, hence $\pi_{W}(\bdy_x W) = \emptyset$.

Conversely, if $W$ is not $x$--compatible, then either distinct boundary components of $\bdy W$ have $x$--representatives that meet `pinching' $W$ or the $x$--representative of some component of $\bdy W$ meets a puncture interior to $W$. In both cases, there is a saddle connection $\sigma \subset \bdy_x W$ that has essential intersection with $W$. 
\end{proof}

\subsection{Filling subsurfaces}\label{Sec:FillingSubsurfaces}

For any essential subsurface $\psub$ of $\mr S$, let $Y = \Fill_S(\psub)$ be the associated subsurface of $S$ obtained by capping off any components of $\bdy \psub$ that are inessential in $S$, as well as any punctures of $\psub$ that are filled in $S$. 

In general, $\Fill_S(\psub)$ might be inessential in $S$, meaning that it is a disk or boundary parallel annulus. However, the situation is better when $\psub$ is $q$--compatible.
 
 {
 \begin{lemma}[Essential filling]\label{Lem:EssentialFill} \label{Lem:QCompatibleFill}
Let $\psub$ be a proper essential subsurface of $\mr S$ such that $d_{\psub}(\lambda^-,\lambda^+)\ge 4$. Let $Y = \Fill_S(\psub)$. 
Then $Y$ is an essential and $q$--compatible (not necessarily proper) subsurface of $S$, satisfying 
$\int_q \psub \subset \int_q Y$. 
Furthermore, if $Y$ is an annulus, then $\int_q \psub = \int_q Y$.
\end{lemma}


\begin{proof}
By \Cref{Lem:CompatibleProj}, $\psub$ is $q$--compatible. That is, $\int_q \psub$ is 
 an embedded (open) representative of $\psub$ in $\mr S$. Since $\bdy_q \psub$ does not pass through any punctures interior to $\psub$, every interior angle along $\bdy_q \psub$ (as measured in $S$) is at least $\pi$. This is the key fact we will use.
 
Now lift $\int_q \psub$ to the cover $\wt S_Y$ of $S$ associated to $Y = \Fill_S(\psub)$ and let $\ol W$ denote its closure. As above, each interior angle of $\ol W$ in $\wt S_Y$ is at least $\pi$ and so $\wt S_Y$ cannot be a disk. This is a consequence of the combinatorial Gauss--Bonnet formula (see e.g.~\cite[Theorem 3.3]{Rafi:ShortCurves}).
Moreover, for essentially the same reason, each boundary component of $\ol W$ either bounds a disk to the outside of $\ol W$ or is homotopic to a flat geodesic not meeting the interior of $\ol W$. 
In particular, $\ol W \subset Y_q$, proving that $Y$ is essential and $q$-compatible. By taking interiors and projecting to $S$, we also see that $\int_q \psub \subset \int_q Y$.

Moreover, if $Y$ is an annulus, then $Y_q$ is a flat annulus. Since $\ol W$ has singularities (or punctures) in its boundary, we must have that both $\ol W$ and $Y_q$ are equal to the maximal flat annulus of $\wt S_Y$. 
\qedhere
\end{proof}
}

We remark that when $Y = \Fill_S(\psub)$, pulling curves tight to saddle connections induces a $1$--Lipschitz map $\AC(Y) \to \AC(\psub)$. 
Compare \Cref{Def:QProjection}.

\begin{lemma}[Cutting subsurfaces]\label{Lem:CuttingSubsurfaces}
Let $\psub$ be a proper essential subsurface of $\mr S$ with $d_{\psub} (\lambda^+, \lambda^-) \geq 4$. 
Let $Y = \Fill_S(\psub)$, and suppose that $f$ overlaps $Y$. Then, for any section $T$, there is an edge $\sigma$ such that both $\sigma$ and $f(\sigma)$ cut $\psub$. 
\end{lemma}

In \Cref{Lem:CuttingSubsurfaces}, the hypothesis that $f(Y) \cap Y \neq \emptyset$ is absolutely crucial. By contrast, the hypothesis on distance in $\psub$ is mainly included for convenience. Observe also that the conclusion of the lemma holds trivially if $\psub = \mr S$.

\begin{proof}
By \Cref{Lem:EssentialFill}, $Y$ is an essential (not necessarily proper) subsurface of $S$. 
By construction, $\psub$ can be recovered from $Y$ by removing some inessential disks and peripheral annuli. Thus the hypothesis that 
$f(Y)$ overlaps $Y$ implies that $f(\psub)$ overlaps $\psub$. 

Since $\psub $ is $q$--compatible by  \Cref{Lem:CompatibleProj}, the interior $\int_q \psub $ is a copy of the interior of $\psub $, with boundary consisting of the saddle connections of $\partial_q \psub$. 

Now, let $T$ be a section, let $t$ be a triangle of $T$, and consider the intersection $t \cap \int_q \psub$. Since $\bdy t$ consists of veering edges, and $\bdy_q \psub $ consists of saddle connections, every component of $t \cap \int_q \psub$ is a $q$--convex polygon. Hence, the intersection of $T$ with $\int_q \psub$ is a cellulation of $\int_q \psub$.

Since $f(\int_q \psub) \cap \int_q \psub \neq \emptyset$, there must be at least one triangle $t \subset T$ such that $f(t \cap \int_q \psub) \cap \int_q \psub \neq \emptyset$. Thus $t \cap \int_q \psub$ and $f(t) \cap \int_q \psub$ are both non-empty polygons, hence there is a veering edge $\sigma \subset t$ such that $\sigma \cap \int_q \psub$ and $f(\sigma) \cap \int_q \psub$ are both non-empty. This edge $\sigma$ has the required property.
\end{proof}

\subsection{Constructing geodesic paths in a subsurface}
The point of this section is to make rigorous the following heuristic: subsurfaces with large projection distance are represented by embedded simplicial `pockets' of the veering triangulation; furthermore, sections sweeping through these pockets closely track geodesics in the curve and arc graph of the subsurface. Since the details will be crucial here, we explain how these properties can be obtained from those in \cite{MinskyTaylor:FiberedFaces}.

For the next statement, if $a$ and $b$ are veering edges, then we say that $a$ \define{lies below} $b$ and write $a<b$ if $a$ crosses the spanning rectangle $R_b$ from left to right (or equivalently, $b$ crosses the spanning rectangle $R_a$ from top to bottom). We will also sometimes say that $b$ is \define{more vertical} than $a$ or that $a$ is \define{more horizontal} than $b$. Compare \Cref{Rem:AboveBelow}.

\begin{lemma}\label{lem:above/below}
Let $\psub$ be a proper essential subsurface of $\mr S$, let $h$ be a veering edge, and let $R_h$ be the singularity-free rectangle spanned by $h$.
\begin{enumerate}
\item
\label{Itm:qAbove} If there is a saddle connection $\sigma \subset \bdy_q \psub$ such that $\sigma < h$, 
then $d_{\psub}(h,\lambda^+) \leq 3$. If $h < \sigma$, then $d_{\psub}(h,\lambda^-) \leq 3$.
\item
\label{Itm:tauAbove} If there is a veering edge $e \subset \bdy_\tau \psub$ such that $e < h$, then $d_{\psub}(h,\lambda^+) \leq 3$. Similarly, if $h < e$, then $d_{\psub}(h,\lambda^-) \leq 3$.
\end{enumerate}
In all cases, there exists a lift of $h$ to $\wt S_{\psub}$ that is disjoint from some  leaf $\ell^\pm \subset \lambda^\pm$. 
\end{lemma}

See \cite[Lemma 6.3]{MinskyTaylor:FiberedFaces} for a very similar statement, with the extra hypothesis that $\psub$ is $\tau$--compatible.

\begin{proof}

To prove \eqref{Itm:qAbove}, let $\wt S_{\psub}$ be the cover of $\mr S$ corresponding to $\psub$, and let $\psub_q \subset \wt S_{\psub}$ be the $q$--hull representative of $\psub$ constructed above. Recall that every component of $\wt S_{\psub} \setminus \psub_q$ is an annulus or disk, called an outward side.
By $q$--convexity, any flat geodesic in $\wt S_{\psub}$ meets $\psub_q$ in a connected subset. 

By hypothesis, there is a saddle connection $\sigma \subset \bdy_q \psub$ that crosses $R_h$ from left to right. In the cover $\wt S_{\psub}$, there is a lift $\widetilde \sigma$ that is a subset of $\bdy \psub_q$, and a lift $\widetilde h$ that intersects 
$\widetilde \sigma$.
Then $\widetilde h$  has at least one endpoint $p_+$ whose neighborhood lies in an outward side of $\wt S_{\psub} \setminus \psub_q$.
Let $p_-$ be the endpoint of $\widetilde h$ opposite $p_+$.
Let  $\ell^+$ be the leaf of $\lambda^+$ starting at $p_-$ and containing the associated side of $R_{\widetilde h}$. (Compare \cite[Figure 22]{MinskyTaylor:FiberedFaces}.)


We claim that $R_{\wt h}$ is embedded in the cover $\wt S_{\psub}$ (rather than merely immersed).
For this, orient $\ell^+$ from $p_-$ and note that by the convexity of $\psub_q$, $\ell^+$ can intersect $\wt \sigma$ at most once with each sign (otherwise, there would be a subarc of $\ell^+$ contained in $\wt S_{\psub} \setminus \psub_q$ with endpoints in $\wt \sigma$). Now, if $R_{\wt h}$ is not embedded, its immersion into $\wt S_{\psub}$ would produce two intersections of $\ell^+$ with $\wt \sigma$ that have the same sign -- a contradiction. Therefore, $\widetilde h$ is the diagonal of an embedded rectangle $R_{\wt h}$, and is disjoint from the leaf $\ell^+$ containing a vertical side.


Since $\widetilde h$ is a component of $\pi_{\psub}(h)$, and $\ell^+$ is a component of $\pi_{\psub}(\lambda^+)$, we conclude that $d_{\psub}(h,\lambda^+) \leq 3$, 
according to the definition of \Cref{Sec:Projection}.
This proves \eqref{Itm:qAbove}.

To prove \eqref{Itm:tauAbove}, let $e \subset \bdy_\tau \psub$ be a veering edge such that $e < h$. Then $e$ crosses the spanning rectangle $R_h$  horizontally, from left to right. Since $\bdy_\tau \psub = \r(\bdy_q \psub)$, there is a saddle connection $\sigma \subset \bdy_q \psub$ such that $e \in \r(\sigma)$. Thus, by definition of $\r$, the edge $e$ is inscribed in a singularity-free rectangle $R_e$ whose diagonal is a segment of $\sigma$. This diagonal crosses $R_h$ from left to right, hence $\sigma$ also crosses $R_h$ from left to right, implying  $\sigma < h$. Thus, by \eqref{Itm:qAbove}, we conclude that $d_{\psub}(h,\lambda^+) \leq 3$.
\end{proof}

For any set $K$ of noncrossing singularity-free saddle connections, let $T(K)$ denote the set of sections containing $K$. 
This set is nonempty by \cite[Lemma 3.2]{MinskyTaylor:FiberedFaces}. Letting $T^- = T^-(K)$ and $T^+ = T^+(K)$ be the lowest and highest sections in $T(K)$, respectively, \cite[Proposition 3.3]{MinskyTaylor:FiberedFaces} ensures that there is a sequence of upward edge flips from $T^-$ to $T^+$ through sections in $T(K)$. 

\begin{lemma}\label{Lem:TopOfPocket}
Let $\psub \subset \mr S$ be a proper essential subsurface. Then every section $T \notin T(\bdy_\tau \psub)$ satisfies $d_{\psub}(T, \lambda^+) \leq 3$ or $d_{\psub}(T, \lambda^-) \leq 3$.
In addition, the top section $T^+ = T^+(\bdy_\tau \psub)$ has the following properties:
\begin{enumerate}
\item Every edge $e_+ \subset T^+ \setminus \bdy_\tau \psub$ satisfies $d_{\psub}(e_+, \lambda^+) \leq 3$.
\item The full section $T^+$ satisfies $d_{\psub}(T^+, \lambda^+) \leq 4$.
\end{enumerate}
Furthermore, the same properties hold with $(T^+, \lambda^+)$ replaced by $(T^-, \lambda^-)$.
\end{lemma}

\begin{proof}
Consider a section $T \notin T(\bdy_\tau \psub)$, and let $h$ be an edge of $T$ that crosses $\partial_\tau \psub$.
 If $h > \bdy_\tau \psub$, then  \Cref{lem:above/below} says that $h$ has a lift $\widetilde h \subset \wt{S}_{\psub}$ that is disjoint from a leaf $\ell^+$ of $\lambda^+$. Since every component of $\pi_{\psub}(T)$ is disjoint from $\widetilde h$ (up to isotopy), and every component of $\pi_{\psub}(\lambda^+)$ is disjoint from $\ell^+$, we conclude that $d_{\psub}(T, \lambda^+) \leq 3$. 
Similarly, if $h < \bdy_\tau \psub$, then $T$ must satisfy  $d_{\psub}(T, \lambda^-) \leq 3$.

Now, let $T^+ = T^+(\bdy_\tau \psub)$ be the top section of $T(\bdy_\tau \psub)$. 
Observe that $T^+$ is a single flip away from a section $T \notin T(\bdy_\tau \psub)$, and the edge being flipped must belong to $\bdy_\tau \psub$. Thus any edge of $T^+$ that is not being flipped --- in particular, any edge of $T^+ \setminus \bdy_\tau \psub$ --- belongs to a section $T$ that satisfies $d_{\psub}(T, \lambda^+) \leq 3$. This implies $d_{\psub}(e_+, \lambda^+) \leq 3$. Since some edge $e_+ \subset T^+$ is not flipped, the full section $T^+$  satisfies $d_{\psub}(T^+, \lambda^+) \leq 4$.
\end{proof}

\begin{theorem}
\label{Thm:NicePath}
Let $\psub$ be a proper essential subsurface of  $\mr S$. Then there is a geodesic $g = \{ e_0, \ldots, e_n \} \subset \AC(\psub)$ that starts within distance $3$ of $\pi_{\psub}(\lambda^-)$ and ends within distance $3$ of $\pi_{\psub}(\lambda^+)$, such that the following holds:
\begin{enumerate}
\item\label{Itm:ProjectedGeo} For any section $T$, we have $d_{\psub}(T, e_i) \leq 5$ for some $e_i \in g$. 
If $T \in T(\bdy_\tau \psub)$, then $d_{\psub}(T, e_i) \leq 2$ for some $e_i \in g$. 
\item\label{Itm:CloseToSomeLam} For any section $T$, we have $\min \{ d_{\psub}(T, \lambda^+), d_{\psub}(T, \lambda^-) \} \leq \frac{1}{2}d_{\psub}(\lambda^-, \lambda^+) + 8$.
\item\label{Itm:VeeringEfficiency} For any pair of sections $T_a$ and $T_b$, we have $d_{\psub}(T_a, T_b) \leq d_{\psub}(\lambda^-,\lambda^+)+13$. If $\psub$ is not $\tau$--compatible, then $d_{\psub}(T_a,T_b)\le 9$.
\end{enumerate}    
\end{theorem}

Although we will not need this, it is also true  that the geodesic $g$ consists of (projections of) veering edges.

It follows from the basic structure of veering triangulations that for any subsurface $\psub$, all sufficiently high sections project to the $1$--neighborhood of $\pi_{\psub}(\lambda^+)$ and sufficiently low sections project to the $1$--neighborhood of $\pi_{\psub}(\lambda^-)$. 

\begin{proof}  
Let $T^- = T^-(\bdy_\tau \psub)$ and $T^+ = T^+(\bdy_\tau \psub)$ be the bottom and top sections in $T(\bdy_\tau \psub)$.
By \Cref{Lem:TopOfPocket}, there exist edges $e_- \subset T^-$ and $e_+ \subset T^+$ such that $d(e_\pm, \lambda^\pm) \leq 3$.
To prove the theorem, we will make an appropriate choice of $e_- \subset T^-$ and $e_+ \subset T^+$, and then construct a geodesic $g \subset \AC(\psub)$ from $e_-$ to $e_+$. Regardless of the choice of $e_\pm$, we make the following observation. If $h$ is a veering edge such that $h < \bdy_\tau \psub$ or $h > \bdy_\tau \psub$, then \Cref{lem:above/below} says that $h$ has a lift $\widetilde h$  that is disjoint from some leaf of $\lambda^\pm$. Since $d(e_\pm, \lambda^\pm) \leq 3$, it will follow that $d_{\psub}(h, e_\pm) \leq 5$.

\smallskip

Suppose that $\psub$ is not $\tau$-compatible. Then we choose arbitrary veering edges $e_- \subset T^- \setminus \bdy_\tau \psub$ and $e_+ \subset T^+ \setminus \bdy_\tau \psub$ such that $d(e_\pm, \lambda^\pm) \leq 3$. Since $\psub$ is not $\tau$--compatible,
 \Cref{lem:stuck} implies there is a saddle connection $\sigma \subset \bdy_\tau \psub$ such that $\pi_{\psub}(\sigma) \neq \emptyset$.
 If $d_{\psub}(e_-, e_+) \leq 1$, we let $g = \{e_-, e_+\}$. Otherwise, since  $\sigma \subset \bdy_\tau \psub$ is disjoint from $e_\pm$, it follows that $g = \{e_-, \sigma, e_+\}$ is a geodesic of length $2$.
 
 Now, we claim that every veering edge $h$ projects within distance $5$ of one of $e_\pm$. If $h < \bdy_\tau \psub$ or $h > \bdy_\tau \psub$, we have already checked that $d_{\psub}(h, e_\pm) \leq 5$. Meanwhile, edges disjoint from $\bdy_\tau \psub$ must be disjoint from $\sigma$, hence 
$d_{\psub}(h, e_\pm) \leq 2$.
 This proves conclusion \eqref{Itm:ProjectedGeo} when $\psub$ is not $\tau$-compatible.
 
 For conclusion \eqref{Itm:CloseToSomeLam}, consider a section $T$. If $T \notin T(\bdy_\tau \psub)$, then \Cref{Lem:TopOfPocket} implies  $\min \{ d_{\psub}(T, \lambda^+), d_{\psub}(T, \lambda^-) \} \leq 3$. Otherwise, if $T \in T(\bdy_\tau \psub)$, then $\sigma \subset T$, hence $d_{\psub}(T, e_-) \leq 2$ and  $d_{\psub}(T, \lambda_-) \leq 5$.
 
For conclusion \eqref{Itm:VeeringEfficiency}, let $T_a$ and $T_b$ be a pair of sections. Consider three cases.
 If neither $T_a$ nor $T_b$ belongs to $T(\bdy_\tau \psub)$, then each one projects within distance $3$ of some $\lambda^\pm$, hence $d_{\psub}(T_a, T_b) \leq d_{\psub}(\lambda^-, \lambda^+) + 6$, which implies $d_{\psub}(T_a, T_b) \leq 9$ by \Cref{Lem:CompatibleProj}. If exactly one of $T_a, T_b$ belongs to $T(\bdy_\tau \psub)$, suppose for concreteness that $T_a$ contains an edge $h_a < \bdy_\tau \psub$ and $T_b \in T(\bdy_\tau \psub)$. Then $d_{\psub}(h_a, e_-) \leq 5$ and $d_{\psub}(e_-, T_b) \leq 2$, hence $d_{\psub}(h_a, T_b) \leq 7$ and $d_{\psub}(T_a, T_b) \leq 8$. Finally, if both of $T_a, T_b$ belong to $T(\bdy_\tau \psub)$, then $\sigma \subset T_a \cap T_b$, hence $d_{\psub}(T_a, T_b) \leq 2$. This proves conclusion \eqref{Itm:VeeringEfficiency} when $\psub$ is not $\tau$-compatible.

\smallskip

Now assume $\psub$ is $\tau$--compatible. Let $e_-$ be an edge of $T^-$ in $\int_\tau \psub$, and recall that $d_{\psub}(e_-, \lambda^-) \leq 3$. Similarly, let $e_+$ be an edge of $T^+$ in $\int_\tau \psub$, and recall that $d_{\psub}(e_+, \lambda^+) \leq 3$.
Set $n = d_{\psub}(e_-, e_+)$, and let $a_0, \ldots, a_n$ be a geodesic in $\AC(\psub)$ from $a_0 = e_-$ to $a_n =e_+$.
We will perturb each $a_i$ to an essential arc or curve $a_i'$ in $\mr S$, as follows. If $a_i$ is a closed curve, let $a_i' = a_i$. If $a_i$ is an arc, then adjust $a_i$ along $\bdy_\tau \psub$ to an arc $a_i' \subset \mr S$  that begins and ends at singularities in the frontier of $\int_\tau \psub$. This adjustment can be done so that $a_i'$ and $a_{i+1}'$ are still disjoint (proceed inductively, starting at $a_1$, and noting that $e_- = a_0$ and $e_+ = a_n$ do not need to be adjusted).  Thus $a_0, a_1', \ldots, a_{n-1}', a_n$ is a path through $\AC(\psub)$, which is geodesic because $d_{\psub}(e_-, e_+) = n$.

Now, apply the coarse $1$--Lipschitz retraction $\p \from \AC(\mr S) \to \A(\tau)$. Since $a_i'$ and $a_{i+1}'$ are disjoint, it follows that $\p(a_i' \cup a_{i+1}')$ has diameter $\leq 1$. Since $\p(\bdy_\tau \psub) = \bdy_\tau \psub$ and $\p$ preserves noncrossingness, it follows that $\p(a_i')$ is a collection of edges of $\int_\tau \psub \bigcup \bdy_\tau \psub$.
For essentially the same reasons, $\p(a_i')$ is essential in $\psub$ (this is only to say that $\p(a_i')$ cannot be entirely contained in $\bdy_\tau \psub$). We may thus select a veering edge $e_i \subset \p(a_i') \setminus \bdy_\tau \psub$. By construction, $e_i$ is disjoint from $e_{i+1}$, and $\{e_0, \ldots, e_n\}$ is a path in $\AC(\psub)$ from $e_- = a_0 = e_0$ to $e_+ = a_n = e_n$. This path of veering edges forms a geodesic in $\AC(\psub)$, because $d_{\psub}(e_-, e_+) = n$. This is our geodesic $g$.


Now let $T_i\in T(\bdy_\tau \psub)$ be a section through $e_i, e_{i+1}$ (that is, $T_i \in T(\bdy_\tau \psub \cup e_i \cup e_{i+1})$). 
Hence $T^-, T_0, \ldots, T_{n-1}, T^+$ project under $\pi_{\psub}$ to the $1$--neighborhood of the geodesic path $g = \{e_0, \ldots, e_n\}$. Moreover, thinking of these sections as lying in $(\mr S\times \RR,\tau)$, adjacent sections have an edge in common and therefore determine a connected chain of sections from $T^-$ to $T^+$. Hence any other section $T \in T(\bdy_\tau \psub)$ must intersect this chain (in a veering edge) and hence projects within a $2$--neighborhood of our geodesic $g$. Meanwhile, any edge $h$ that lies above or below $\bdy_\tau \psub$ must project within distance $5$ of one of $e_+$ or $e_-$, completing the proof of  \eqref{Itm:ProjectedGeo}.

 For conclusion \eqref{Itm:CloseToSomeLam}, consider a section $T$. If $T \notin T(\bdy_\tau \psub)$, then \Cref{Lem:TopOfPocket} implies  $\min \{ d_{\psub}(T, \lambda^+), d_{\psub}(T, \lambda^-) \} \leq 3$. Otherwise, we have just shown that $d_{\psub}(T, e_j) \leq 2$ for some veering edge $e_j$ in the geodesic $g$. Thus some endpoint of $g$ (say, $e_-$ for concreteness) satisfies $d_{\psub}(T, e_-) \leq \frac{1}{2}d_{\psub}(e_-, e_+) + 2$, and we conclude
 \[
 d_{\psub}(T, \lambda_-) \leq d_{\psub}(T, e_-) + 3 \leq \tfrac{1}{2}d_{\psub}(e_-, e_+) + 5 \leq \tfrac{1}{2} \left( d_{\psub}(\lambda^-, \lambda^+) + 6 \right) + 5,
 \]
 where the third inequality holds because $d_{\psub}(e_\pm, \lambda^\pm) \leq 3$.

For conclusion \eqref{Itm:VeeringEfficiency}, let $T_a$ and $T_b$ be a pair of sections. We consider two cases. If neither $T_a$ nor $T_b$ belongs to $T(\bdy_\tau \psub)$, then each one projects within distance $3$ of some $\lambda^\pm$, hence $d_{\psub}(T_a, T_b) \leq d_{\psub}(\lambda^-, \lambda^+) + 6$. Alternately, if $T_b  \in T(\bdy_\tau \psub)$, then $\pi_{\psub}(T_b)$ lies within distance $2$ of the geodesic $g = \{e_0, \ldots, e_n\}$, while $T_a$ lies within distance $5$ of $g$. Thus
\[
d_{\psub}(T_a, T_b) \leq d_{\psub}(e_-, e_+) + 7 \leq \left( d_{\psub}(\lambda^-, \lambda^+) + 6 \right) +7.
\]
 This proves conclusion \eqref{Itm:VeeringEfficiency} and the theorem.
\end{proof}

\begin{remark}\label{Rem:NonBacktracking}
The proof of \Cref{Thm:NicePath} can be adapted to prove the following ``non-backtracking statement''. Suppose that $T_a, T_b \in T(\bdy_\tau \psub)$ are a pair of sections such that $T_a < T_b$. Let $g = \{e_0, \ldots, e_n\}$ be the geodesic of \Cref{Thm:NicePath}, so that $e_a$ is the closest vertex to $T_a$ and $e_b$ is the closes to $T_b$. Then the entire upward flip sequence from $T_a$ to $T_b$ projects within distance $2$ of the sub-geodesic $\{e_a, \ldots, e_b\}$.
 
If $\psub$ is not $\tau$--compatible, then  the entire upward flip sequence from $T_a$ to $T_b$ must contain the saddle connection  
$\sigma \subset \bdy_\tau \psub$ such that $\pi_{\psub}(\sigma) \neq \emptyset$. Thus the flip sequence projects within distance $1$ of $\sigma$.

Now, we may assume $\psub$ is $\tau$--compatible. As in the above proof, the section $T_a$ must intersect some $T_i \in T(\bdy_\tau \psub \cup e_i \cup e_{i+1})$ in a veering edge, hence 
\[ d_{\psub}(T_a, e_i \cup e_{i+1}) \leq 2.\]
 Similarly, $T_b$ must intersect some 
$ T_j \in T(\bdy_\tau \psub \cup e_j \cup e_{j+1})$
  in a veering edge. Now, every $T'$ in the flip sequence must intersect one of the $T_k$'s for $i \leq k \leq j$.
\end{remark}

\subsection{Sections and intersections}\label{Sec:SectionsIntersections}
Recall that a section  $T$ can simultaneously be thought of as an ideal triangulation of $\mr S$ by veering edges, and as a copy of $\mr S$ simplicially embedded in 
the veering triangulation $\tau_X$ on $\mr S \times \RR$. In what follows, we make use of both perspectives.

Now, suppose that $T_a$ and $T_b$ are a pair of sections. Then $T_a$ and $T_b$ partition $\mr S \times \RR$ into two unbounded components and some number of bounded components. Each bounded component of $\mr S \times \RR \setminus (T_a \cup T_b)$ is called a \define{pocket}. Each pocket $P$ projects to a region $B = B(P) \subset \mr S$  called the \define{base} of the pocket, whose boundary is a union of veering edges. The top of the pocket $P$, denoted $P^+$, is the sub-section of one section ($T_a$ or $T_b$) lying over $B$. The bottom of the pocket, denoted $P^-$, is a subsection of the other section lying over $B$. 

For a singularity-free saddle connection $\sigma$, set $T^+_\sigma = T(\sigma)^+$ to be the highest section in $T(\sigma)$, meaning the highest section containing $\sigma$ as an edge. 
This can be characterized by the property that if $T' \in T(\sigma)$ is any other veering triangulation of $\mr S$ containing $\sigma$, then for any crossings between edges of $T^+_\sigma$ and $T'$, the edges of $T^+_\sigma$ are more vertical. See \cite[Section 3]{MinskyTaylor:FiberedFaces} for more details.

Similarly, let $T^-_\sigma = T(\sigma)^-$ be the lowest section containing $\sigma$.

\begin{definition}\label{Def:PartialSections}
Suppose that $\sigma_1$ and $\sigma_2$ are crossing singularity-free saddle connections, such that $\sigma_1 > \sigma_2$. 
In this case, the sections $T^-_{\sigma_1}$ and $T^+_{\sigma_2}$ together cut $\mr S \times \RR$ into  two unbounded regions and a collection of pockets. 
Let $P = P(\sigma_1, \sigma_2)$ be the unique pocket containing $\sigma_1$ on its top side and $\sigma_2$ on its bottom side. The top and bottom of the pocket are denoted $P(\sigma_1, \sigma_2)^+$ and $P(\sigma_1, \sigma_2)^-$, respectively.
Observe that both $P(\sigma_1, \sigma_2)^\pm$ are triangulations of $B$, the base of the pocket, and are called the \define{partial sections} defined by $\sigma_1, \sigma_2$.
As triangulations of $B$, we have
\begin{equation}\label{Eqn:ReversePlusMinus}
P(\sigma_1, \sigma_2)^+ = T^-_{\sigma_1} \cap B \qquad  \text{ and } \qquad P(\sigma_1, \sigma_2)^- =  T^+_{\sigma_2} \cap B.
\end{equation}
(Note the somewhat counterintuitive interchange of pluses and minuses.) 
In the important special case where $\sigma_1$ and $\sigma_2$ fill $\mr S$, we have $T^-_{\sigma_1} > T^+_{\sigma_2}$, and the two disjoint sections co-bound a single pocket over $B = \mr S$.
\end{definition}


\begin{lemma}[Closest sections]\label{Lem:ClosestPartialSections}
Let $\sigma_1$ and $\sigma_2$ be crossing singularity-free saddle connections such that   $\sigma_1  > \sigma_2$.
Let $P(\sigma_1, \sigma_2)$ be the pocket defined by $\sigma_1$ and $\sigma_2$, with top and bottom sides
$P_1 = P(\sigma_1, \sigma_2)^+$ and $P_2 = P(\sigma_1, \sigma_2)^-$.
Then 
\[
 i(\sigma_1,\sigma_2) \: \leq \: i(P_1,P_2) \: \leq \: (9 \chi(S))^2 \cdot i(\sigma_1,\sigma_2).
 \]
\end{lemma}

\begin{proof}
First, we check that the number of edges in any section $T$ is at most $|9 \chi(S)|$. Since $T$ is an ideal triangulation of the punctured surface $\mr S$, it has exactly $ |3 \chi(\mr S)|$ edges. Furthermore, since $\mr S$ is created by puncturing the singularities of $q$, and the number of singularities is largest when all interior singularities are trivalent, the Euler--Poincar\'e formula implies 
$\# \sing(q) \leq 2 |\chi(S)|$, hence
$|\chi(\mr S)| \leq 3 |\chi(S)|$. Putting it all together shows the number of edges is at most $|9 \chi(S)|$. 

Now, we prove the lemma. The first inequality in the statement is immediate by set inclusion, since $\sigma_j \subset P_j$.

For the second inequality, the key idea is contained in the proof of \cite[Lemma 3.4]{MinskyTaylor:FiberedFaces}. 
Recall from 
\eqref{Eqn:ReversePlusMinus} that $P_1 = T^-_{\sigma_1} \cap B$  and $P_2 =  T^+_{\sigma_2} \cap B$.
By construction, since $T^-_{\sigma_1}$ is the lowest section containing $\sigma_1$, the only  downward flippable edge of $T^-_{\sigma_1}$ is $\sigma_1$.  But \cite[Lemma 3.4]{MinskyTaylor:FiberedFaces} shows that for any edge $e \subset P_2$, the tallest edge of $T^-_{\sigma_1}$ crossing $e$ is always downward flippable. Thus, for every $e \subset P_2$, the tallest edge of $P_1$ crossing $e$ must be $\sigma_1$. 

From this, it follows that every edge $h \subset P_1$ satisfies $i(h,e) \le i(\sigma_1,e)$. Indeed, starting at $h$ there is a sequence $h_0, h_1, \ldots, h_k$ of edges of $P_1$ with $h_0=h$, $h_k = \sigma_1$, and, for each $i =0 ,\ldots, k-1$, $h_i$ and $h_{i+1}$ cobound a triangle for which $h_{i+1}$ is the tallest edge. 
Fixing a lift $\wt e$ of $e$ to the universal cover, each intersection of $h$ with $e$ determines a unique lift $\wt h$ of $h$ that intersects $\wt e$. The above sequence has a unique lift staring at $\wt h$, and the discussion above implies that each edge along the lifted sequence must intersect $\wt e$. Moreover, using uniqueness of the lifted path, it is easy to see that the associated map from $h \cap e$ to $\sigma_1 \cap e$ is injecitve.


As the number of edges  $h \subset P_1$ is at most $|9 \chi(S)|$,
 we obtain 
\[
i(P_1,e) \leq |9 \chi(S)| \cdot i(\sigma_1,e)
\quad \Rightarrow \quad
i(P_1, P_2) \leq |9 \chi(S)| \cdot i(\sigma_1, P_2).
\]
Now, performing the same analysis for $\sigma_2$ in place of $\sigma_1$ gives
\[
i(\sigma_1, P_2) \leq |9 \chi(S)| \cdot i(\sigma_1, \sigma_2),
\]
completing the proof.
\end{proof}

\begin{lemma}
\label{Lem:ClosestSecSubsurfaceDist}
Let $T$ be a section of the veering triangulation and $\sigma$ an edge of $T$. 
Assume that $i(\sigma, f(\sigma)) > 0$. Let $P = P(\sigma, f(\sigma))$ be the associated pocket, with top and bottom sides $P^+$ and $P^-$.
Then
for any essential subsurface $\psub \subset \mr S$,
\[
d_{\psub}(\sigma , f(\sigma)) \: \leq \: d_{\psub}(P^+, P^-) \: \leq \:  d_{\psub}(\sigma , f(\sigma)) + 9.
\]
\end{lemma}

\begin{proof}
Let $B$ be the base of the pocket $P = P(\sigma, f(\sigma))$. If $B \cap \psub = \emptyset$, then $\pi_{\psub}(P^+) = \pi_{\psub}(P^-) = \emptyset$. In this case,
 the statement is immediate because all distances in the lemma are $\diam(\emptyset) = 0$. Regardless of intersection $B \cap \psub$, the first inequality is also immediate because $\sigma \subset P(\sigma, f(\sigma))^+$ and $f(\sigma) \subset P(\sigma, f(\sigma))^-$. 


For the second inequality, we may assume that  $\psub$ is $\tau$--compatible: otherwise, 
\Cref{Thm:NicePath} implies $d_{\psub}(P^+, P^-) \leq 9$ because  $P^\pm$ are partial sections.

If $\sigma$ and $f(\sigma)$ both cut $\psub$, observe that $\pi_{\psub}(P^+)$ lies in a $1$--neighborhood of  $\pi_{\psub}(\sigma)$, and $\pi_{\psub}(P^-)$ similarly lies in a $1$--neighborhood of $f(\sigma)$. Thus we have $d_{\psub}(P^+, P^-) \leq d_{\psub}(\sigma , f(\sigma)) + 2$.

It remains to prove the second inequality in the special case where $\psub$ is $\tau$--compatible, $\psub \cap B \neq \emptyset$, and one of $\sigma$ and $f(\sigma)$ misses $\psub$. We assume for concreteness that $\sigma \cap \psub = \emptyset$. In particular, $\mr W$ is a proper subsurface.

Since the subsurface $\psub$ is $\tau$--compatible, it is isotopic to the base $\psub_\tau$ of a pocket $P(\psub_\tau)$. Since $\sigma$ misses $\psub$, the section $T^-_\sigma$ lies no higher than the bottom side of $P(\psub_\tau)$. 
This is to say that either some edge of $T^-_\sigma$ lies below an edge of $\partial_\tau \psub$, or all of $\partial_\tau \psub$ belongs to $T^-_\sigma$. In the first case, we have $d_{\psub} (T^-_\sigma, \lambda^-) \leq 3$ by \Cref{Lem:TopOfPocket}, and in the second case we have $d_{\psub} (T^-_\sigma, \lambda^-) \leq 4$ by applying \Cref{Lem:TopOfPocket} to a section that lies one flip below $T^-_\sigma$. 
Since $P^+ \subset T^-_\sigma$ by \eqref{Eqn:ReversePlusMinus},
we have $d_{\psub} (P^+, \lambda^-) \leq 4$.

Finally, recall that $f(\sigma) \subset P^-$, which lies below $P^+$.
Thus some section containing $P^-$ must also be within distance $4$ of $\lambda^-$. We conclude that $d_{\psub} (P^-, \lambda^-) \leq 4$, hence $d_{\psub}(P^+, P^-) \leq 8.$
\end{proof}

\section{Isolated pockets and annular-avoiding sections}\label{sec:aas}

In distance formulas in Teichm\"uller space (such as \Cref{Prop:RafiInterpreted} and \Cref{Prop:ChoiRafi}), annular projections play a special role. 
The special role of annuli is also visible in \Cref{th:main_2}.
In this section of the paper, we will carefully build sections of the veering triangulation that avoid passing through deep annular pockets. This will enable our formulas to correctly count annuli.

We will also prove that for each (proper) subsurface $\psub$, a well-chosen section $T$ satisfies a coarse equality of the form $d_{\psub}(\lambda^+, \lambda^-) \asymp d_{\psub}(T, f(T))$. These coarse equalities are recorded in \Cref{Lem:AnnularAvoiding,Lem:SectionsAlongGeodesic,Lem:fSectionSum}.

\begin{definition}\label{Def:IsolatedPocket}
Let $\psub \subset \mr S$ be a $\tau$--compatible subsurface. A pocket $V = V_{\psub}$ over $\int_\tau \psub$ is called \define{isolated} if the following properties hold:
\begin{itemize}
\item Every edge $e \subset V_{\psub} \setminus \bdy_\tau \psub$ satisfies $d_{\psub}(e, \lambda^-) \geq 6$ and $d_{\psub}(e, \lambda^+) \geq 6$.
\item The top and bottom boundaries of $V_{\psub}$, denoted $V^+$ and $V^-$, intersect only in $\bdy_\tau \psub$.
\end{itemize}
\end{definition}

We remark that \Cref{Def:IsolatedPocket}  differs slightly from the notion of isolated pockets in \cite[Section 6.3]{MinskyTaylor:FiberedFaces}. This is because Minsky and Taylor \cite{MinskyTaylor:FiberedFaces}  use a different convention for subsurface distances: to compute $d_{\psub}(\alpha, \beta)$, they consider the minimal distance between vertices in $\pi_{\psub}(\alpha)$ and $\pi_{\psub}(\beta)$ rather than the maximal distance. After adjusting for the slightly different distance convention, it follows that \Cref{Def:IsolatedPocket} is slightly stronger.

\begin{lemma}[Isolated pockets]\label{Lem:IsolatedPocket}
For every proper essential subsurface $\psub \subset \mr S$ such that $d_{\psub}(\lambda^-, \lambda^+) > 12$, there is an isolated pocket $V_{\psub}$ over $\psub_\tau$. This pocket has the following properties:
\begin{enumerate}
\item\label{Itm:SectionDistBelow} Every section $T$ that lies below the interior of $V_{\psub}$ satisfies $d_{\psub}(T, \lambda^-) \leq 6$.
Every section $T$ that lies above the interior of $V_{\psub}$ satisfies $d_{\psub}(T, \lambda^+) \leq 6$.
\item\label{Itm:IsolatedDisjointness} If $V_{\psub}$ and $V_{\mr Y}$ are isolated pockets whose bases $\psub$ and ${\mr Y}$ are overlapping proper subsurfaces, then a section $T$ cannot intersect the interiors of both $V_{\psub}$ and $V_{\mr Y}$. In particular, these interiors are disjoint.
\end{enumerate}
\end{lemma}

\begin{proof}
Conclusion \eqref{Itm:SectionDistBelow} follows from the argument used to prove \cite[Lemma 6.4]{MinskyTaylor:FiberedFaces}. Since the definition has changed slightly, we give a brief summary. 

By \Cref{Lem:CompatibleProj}, $\psub$ is $\tau$--compatible.
Let $T_0 = T^-(\bdy_\tau \psub), T_1, \ldots, T_n = T^+(\bdy_\tau \psub)$ be an upward flip sequence from $T_0 = T^-(\bdy_\tau \psub)$ to $T_n = T^+(\bdy_\tau \psub)$. 
Let $a$ be the largest integer in $\{0, \ldots, n\}$  such that $d_{\psub}(e, \lambda^-) < 6$ for some edge $e \subset T_{a-1}$ that projects to $\psub$; such an $a$ exists by \Cref{Lem:TopOfPocket}. Similarly, let $b$ be the smallest integer in the sequence such that $d_{\psub}(e, \lambda^+) < 6$ for some edge $e \subset T_{b+1}$ projecting to $\psub$. The hypothesis $d_{\psub}(\lambda^-, \lambda^+) > 12$ implies that $a < b$, and that $T_a, T_b$ cannot share any edges over $\psub$.
Now, let $V_{\psub}$ be the pocket over $\psub$ whose bottom side is contained in $T_a$ and whose top side is contained in $T_b$. 
By construction, every edge in $T_a, \ldots, T_b$ that lies over $\psub$ satisfies $d_{\psub}(e, \lambda^-) \geq 6$ and  $d_{\psub}(e, \lambda^+) \geq 6$. Thus $V_{\psub}$ is an isolated pocket, and property \eqref{Itm:SectionDistBelow} also follows.

For conclusion \eqref{Itm:IsolatedDisjointness}, suppose for a contradiction that $T$ intersects the interiors of both $V_{\psub}$ and $V_{\mr Y}$. 
Then $T$ contains an edge $h \subset V_{\mr Y}$ that is disjoint from $\bdy_\tau  {\mr Y}$. Since the subsurfaces $\psub$ and ${\mr Y}$ are $\tau$--compatible and have overlapping boundaries, we may assume for concreteness that
 some edge $h' \subset \bdy_\tau {\mr Y}$ lies below an edge of $\bdy_\tau \psub$. By \Cref{lem:above/below}, we have $d_{\psub}(h', \lambda^-) \leq 3$, which implies $d_{\psub}(h, \lambda^-) \leq 4$ and $d_{\psub}(T, \lambda^-) \leq 5$. On the other hand, since $T$ intersects the interior of $V_{\psub}$, \Cref{Def:IsolatedPocket} implies $d_{\psub}(T, \lambda^-) \geq 6$, which is a contradiction. It follows that the interiors of $V_{\psub}$ and $V_{\mr Y}$ cannot intersect the same section, and are therefore disjoint.
\end{proof}

We remark that the disjointness of isolated pockets (as defined in Minsky--Taylor \cite[Section 6.3]{MinskyTaylor:FiberedFaces}) is proved in \cite[Proposition 6.5]{MinskyTaylor:FiberedFaces}. It is also worth observing that 
the interiors of isolated pockets $V_{\psub}$ and $V_{\mr Y}$ must automatically be disjoint whenever $\psub$ and ${\mr Y}$ have disjoint interiors.

%
%

\begin{definition}
A section $T$ is called an \define{$f$--section} if $f(T) \leq T$. Here the ordering is with respect to $\mr S \times \RR$. Equivalently, if $\sigma_1$ and $\sigma_2$ are edges of $T$ such that $f(\sigma_1)$ and $\sigma_2$ cross, then $f(\sigma_1)$ crosses the spanning rectangle of $\sigma_2$ from left to right. 
\end{definition}

\smallskip
The following lemma essentially states that there exists an $f$--section $T$ that does not cut deeply through any annular pocket.
By \Cref{Lem:LowDegreeAwayFromPocket}, this will imply that any veering edge $\sigma \subset T$ has a spanning rectangle of bounded degree, enabling us to apply the fundamental lemma to count fixed points.

\begin{lemma}[Annular avoiding section]\label{Lem:AnnularAvoiding}
There is an $f$--section $T$ of $\mr S \times \RR$ such that for any essential annulus $A$ of $\mr S$, we have
\begin{equation}\label{Eqn:CloseToLamAnnular}
\min \big\{ d_A(T, \lambda^-), d_A(T, \lambda^+) \big\} \leq 14.
\end{equation}
Furthermore, if $d_{A}(\lambda^+,\lambda^-)\geq 25$, there is a unique $f$--translate $A' = f^i(A)$ such  that $d_{A'}(T, f(T)) > 12$. This translate $A'$ satisfies
\begin{equation}\label{Eqn:SimilarDistAnnular}
 d_{A'}(\lambda^+,\lambda^-) -12 \: \leq \: d_{A'}(T, f(T)) \: \leq \: d_{A'}(\lambda^+,\lambda^-) +13.
\end{equation}
\end{lemma}

An $f$--section $T$ satisfying the conclusion of this lemma is called \define{annular-avoiding}.


\begin{proof}
We begin by proving that there exists a section $T$ satisfying \eqref{Eqn:CloseToLamAnnular} for every annulus $A$.
If $d_A(\lambda^-, \lambda^+) \leq 12$, then \Cref{Thm:NicePath} implies that \emph{every} section $T$ satisfies 
\[
\min \{d_A(T, \lambda^-), d_A(T, \lambda^+)\} \leq \tfrac{1}{2} d_A(\lambda^-, \lambda^+) + 8 \leq 14.
\]
 Thus it suffices to consider annuli with $d_A(\lambda^-, \lambda^+) > 12$.

For any annulus $A$ such that $d_A(\lambda^-, \lambda^+) > 12$, \Cref{Lem:IsolatedPocket} says that there is an isolated pocket $V_A$. Since any pair of annuli $A$ and $A'$ either overlap or are disjoint, \Cref{Lem:IsolatedPocket} says that  the interiors of all isolated annular pockets are pairwise disjoint. Thus we may perform the following procedure. We begin with an arbitrary section $T_0$, and flip $T_0$ down to the bottom of every isolated pocket that it crosses. The resulting section $T$ is disjoint from the interior of every isolated pocket. Thus, by \Cref{Lem:IsolatedPocket}, we have 
\begin{equation}\label{Eqn:PocketBottomClose}
\min \big\{ d_A(T, \lambda^-), d_A(T, \lambda^+) \big\} \leq 6
\end{equation}
for every annulus $A$ with $d_A(\lambda^-, \lambda^+) > 12$. We have thus found a section $T$  (not necessarily an $f$--section) satisfying \eqref{Eqn:CloseToLamAnnular} in every annulus.

To obtain an $f$--section, consider the family of sections $\{ f^i(T): i \geq 0 \}$. For sufficiently large $i$ --- where $i > (9 \chi(S))^2$ suffices --- these sections lie strictly below $T$. Now, we define $T^+$ to be the maximum of this family of sections. (The maximum operation is carefully defined just above \cite[Lemma 3.2]{MinskyTaylor:SubsurfaceDist}, and coincides with the pointwise max along the $\RR$--fibers.)
It is clear that $f(T^+) \leq T^+$, since $f(T^+)$ is the maximum of a smaller family of sections. Moreover, since $T$ is disjoint from the interior of every isolated annular pocket, the same is true for its translates, hence $T^+$ is disjoint from the interior of every isolated pocket as well. Thus $T^+$ satisfies \eqref{Eqn:PocketBottomClose} in every annulus $A$ with $d_A(\lambda^-, \lambda^+) > 12$, and we conclude that $T^+$ is an $f$--section satisfying \eqref{Eqn:CloseToLamAnnular}.

To prove that there is an $f$--translate of $A$ satisfying \eqref{Eqn:SimilarDistAnnular}, let $A$ be an annulus such that $d_A(\lambda^-, \lambda^+) \geq 25$. Consider the family of sections $\{ f^i(T^+): i \in \ZZ \}$. Since $T^+$ is disjoint from (the interior of) every isolated annular pocket, the same is true for its translates. In particular, every $ f^i(T^+)$ is either above or below the isolated pocket $V_A$. Since $f^{i+1}(T^+) \leq f^i(T^+)$ for every $i$, there is a single integer $n$ such that $ f^i(T^+)$ is above $V_A$ for $i \leq n$ and below $V_A$ for $i > n$. Thus, for every $i \neq n$, \Cref{Lem:IsolatedPocket} implies
that $ f^i(T^+)$ and $f^{i+1}(T^+)$ are $6$--close to the same lamination, hence
\[
d_A(f^i(T^+), f^{i+1}(T^+)) \leq 12 <  d_{A}(\lambda^+,\lambda^-) - 12
\quad \text{for} \quad i \neq n.
\]
On the other hand, we have $d_A(f^n(T^+), \lambda^+) \leq 6$ and $d_A(f^{n+1}(T^+), \lambda^-) \leq 6$, hence
\[
d_A(f^n(T^+), f^{n+1}(T^+)) \geq  d_{A}(\lambda^+,\lambda^-) - 12.
\]
Translating everything over by $f^{-n}$, we learn that $A' = f^{-n}(A)$ is the only $f$--translate of $A$ where the projections of $T^+$ and $f(T^+)$ are sufficiently far. Meanwhile, the opposite inequality $ d_{A'}(T, f(T)) \leq d_{A'}(\lambda^+,\lambda^-) +13$, holds for every section by \Cref{Thm:NicePath}.
\end{proof}

Annular-avoiding $f$--sections are needed for two reasons. First, as mentioned above, the edges of an annular-avoiding section $T$
 have bounded degree, enabling us to apply the fundamental lemma. Second, annular projections will be counted with a $\log$ in \Cref{Prop:RafiInterpreted}, so we do not want to break up a big annular projection into two projections. (For non-annular subsurfaces $\psub$, it can indeed happen that a large distance $d_{\psub}(\lambda^+, \lambda^-)$ gets broken up into a large (but bounded) number of large projections; see \Cref{Rem:fSectionSumWithoutOverlap}.)

In the remainder of this section, we prove coarse comparisons similar to \eqref{Eqn:SimilarDistAnnular} that hold for non-annular subsurfaces of $\mr S$.


\begin{lemma}\label{Lem:SectionsAlongGeodesic}
 For any section $T$ of $\mr S \times \RR$, we have
 \[
 d_{\mr S}(T, f(T)) - 4 \: \leq \: \ell_{\mr S}(\mr f) \: \leq \: d_{\mr S}(T, f(T)).
 \]
\end{lemma}

Recall from \eqref{Eqn:StableLength} that $\ell_{\mr S}(\mr f) $ denotes the stable translation length of $\mr f$ in $\AC(\mr S)$.

\begin{proof}
For every $n \in \NN$, triangle inequalities imply that $d_{\mr S}(T, f^n(T)) \leq n \cdot d_{\mr S}(T, f(T))$.
Thus 
\[
\ell_{\AC(\mr S)}(\mr f) = \lim_{n \to \infty} \frac{d(T, f^n(T))}{n}  \leq \lim_{n \to \infty} \frac{n \cdot d_{\mr S}(T, f(T))}{n} = d_{\mr S}(T, f(T)).
\]

For the other inequality, we may assume that $ d_{\mr S}(T, f(T)) > 4$ (otherwise, there is nothing to prove).
By \cite[Theorem 1.4]{MinskyTaylor:FiberedFaces}, there is a bi-infinite geodesic axis $g = \{e_i : i \in \ZZ\}$ that runs through $\AC(\mr S)$ from $\lambda^+$ to $\lambda^-$, such that every $e_i$ is an edge of the veering triangulation. 
 As in the proof of \Cref{Thm:NicePath}, let $T_i$ be a section through $e_i \cup e_{i+1}$. Then every section $T'$ intersects some $T_i$ in a veering edge. In particular, every $T'$ lies within distance $2$ of some $e_i$.
 
 For every $k \in \ZZ$, let $i_k \in \ZZ$ be an integer such that $d_{\mr S}(f^k(T), e_{i_k}) \leq 2$. Then triangle inequalities imply
 \[
 i_{k+1} - i_k = d_{\mr S} (e_{i_k}, e_{i_{k+1}}) \geq d_{\mr S}(f^{k+1}(T), f^k(T)) - 4 = d_{\mr S}(T, f(T)) - 4 > 0.
 \]
 In particular, the veering edges $e_{i_k}$ are monotonically ordered along $g$.
 
 Now, fix an $\epsilon > 0$. Then, for all $n > N(\epsilon)$, we have 
 $\ell_{\AC(\mr S)}(\mr f) \geq \frac{d(T, f^n(T))}{n} - \epsilon$. Furthermore,
  \begin{align*}
d(T, f^n(T))
 &\geq d_{\mr S} (e_{i_0}, e_{i_{n}}) - 4 \\
 & = \sum_{j = 0}^{n-1} d_{\mr S} (e_{i_j}, e_{i_{j+1}}) \quad -4 \\
 & \geq \sum_{j = 0}^{n-1} (d_{\mr S}(T, f(T)) - 4) \quad -4 \\
 & = n (d_{\mr S}(T, f(T)) - 4) \quad -4
 \end{align*}
 which implies  $\ell_{\AC(\mr S)}(\mr f) \geq (d_{\mr S}(T, f(T)) - 4) - \frac{4}{n} - \epsilon$.
Letting $n \to \infty$ and then $\epsilon \to 0$ completes the proof.
\end{proof}

\begin{lemma}\label{Lem:fSectionSum}
 Let $T$ be an $f$--section of $\mr S$. Let
 $\psub \subset \mr S$ be a proper non-annular subsurface such that $d_{\psub}(\lambda^+,\lambda^-) > 12$, and such that $f$ overlaps $\psub$.
 Then there is a sub-surface $\psub'  = f^{n_0}(\psub)$ such that
 \[
d_{\psub}(\lambda^+,\lambda^-) \: = \:  d_{\psub'} (T, f(T)) + d_{f(\psub')} (T, f(T))  + O(1).
\]
Furthermore, $d_{f^i(\psub')}(T, f(T)) \leq 12$ for all $i \neq 0,1$. In particular, for every $K \geq 13$ there is hierarchy-like sum of the form 
\[
\big[ d_{\psub}(\lambda^+, \lambda^-) \big]_{K}  \asymp \sum_{i \in \ZZ} \big[ d_{f^i(\psub)}(T, f(T)) \big]_{K}.
\]
\end{lemma}

\begin{proof}
Since $d_{\psub}(\lambda^+,\lambda^-) > 12$, \Cref{Lem:IsolatedPocket} says there is an isolated pocket $V_{\psub}$ over $\psub$. Thus, for every $i \in \ZZ$, the translate $f^i(V_{\psub})$ is an isolated pocket over $f^i(\psub)$, and these pockets have pairwise disjoint interiors.

First, suppose that $T$ intersects the interior of $V_{\psub'}$ for some translate $\psub' = f^{n_0}(\psub)$.
Since $f(\psub')$ overlaps $\psub'$,  \Cref{Lem:IsolatedPocket} implies that  $T$ cannot intersect the interior of $f(V_{\psub'})$. Since $f$ translates downward, it follows that $f^i(V_{\psub'})$ lies below $T$ for every $i > 0$, hence $T$ and $f(T)$ both lie above $f^i(V_{\psub'})$ for every $i > 1$. By an identical argument, $T$ cannot intersect the  interior of $f^{-1}(V_{\psub'})$, hence $T$ and $f(T)$ both lie below $f^i(V_{\psub'})$ for every $i < 0$. 

Now, suppose that $T$ is disjoint from the interior of every pocket in the $\langle f \rangle$--orbit of $V_{\psub}$. 
Then there is exactly one translate $\psub' = f^{n_0}(\psub)$
such that the interior of  $V_{\psub'}$ lies below $T$ and above $f(T)$. Since $f$ translates downward, it once again follows that $f^i(V_{\psub'})$ lies below $T$ for every $i \geq 0$,  hence $T$ and $f(T)$ both lie above $f^i(V_{\psub'})$ for every $i \geq 1$. Similarly,   $T$ and $f(T)$ both lie below $f^i(V_{\psub'})$ for every $i < 0$. 

Combining the cases, we learn that $T$ and  $f(T)$ both lie on the same side of $f^i(V_{\psub'})$ for every $i \neq 0,1$, hence  \Cref{Lem:IsolatedPocket} implies that $d_{f^i(\psub')}(T, f(T)) \leq 12$.

For the distance estimate of the lemma, observe that  $f^{-1}(T)$ lies above $V_{\psub'}$, hence \Cref{Lem:IsolatedPocket} tells us that  $d_{\psub'}(f^{-1}(T), \lambda^+) \leq 6$. Similarly,  $f(T)$ lies below $V_{\psub'}$, hence \Cref{Lem:IsolatedPocket} tells us that   $d_{\psub'}(f(T),  \lambda^-) \leq 6$. 
%
%
Now, recall the geodesic $g = \{e_0, \ldots, e_n \} \subset \AC(\psub')$ guaranteed by \Cref{Thm:NicePath}. 
By that theorem, the projection to $\psub$ of each of $f^{-1}(T), T, f(T)$ lies within distance $5$ of some vertex of $g$. Furthermore, $\pi_{\psub}(f^{-1}(T))$ is close to $e_0$ and $\pi_{\psub}(f(T))$ is close to $e_n$, because they are close to $\lambda^+$ and $\lambda^-$, respectively. Thus
\begin{align*}
d_{f(\psub')}(T, f(T)) + d_{\psub'}(T, f(T)) & = d_{\psub'}(f^{-1}(T), T) + d_{\psub'}(T, f(T)) \\
& = d_{\psub'} (e_0, e_n) + O(1) \\
& = d_{\psub'} (\lambda^+, \lambda^-) + O(1),
\end{align*}
completing the proof.
\end{proof}

\begin{remark}\label{Rem:fSectionSumWithoutOverlap}
There is a more subtle version of 
 \Cref{Lem:fSectionSum} that holds without the hypothesis that $\psub$ is overlapped. In that setting, there might be multiple sections $T, f^{-1}(T), \ldots, f^{-(k-1)}(T)$ that all intersect the same isolated pocket $V_{\psub'} = f^{n_0}(\psub)$. The number $k$ of these sections is bounded by the number of translates of $\psub$ whose interiors are pairwise disjoint, which is itself bounded by $|\chi(\mr S)|$. The nearest-point projections of $T, f^{-1}(T), \ldots, f^{-(k-1)}(T)$ to the geodesic $g = \{e_0, \ldots, e_n \} \subset \AC(\psub')$ guaranteed by \Cref{Thm:NicePath} are coarsely non-backtracking by \Cref{Rem:NonBacktracking}.
By computing as in \Cref{Lem:SectionsAlongGeodesic}, one can prove that
\[
d_{\psub}(\lambda^+,\lambda^-) = \sum_{i=n_0}^{n_0+k} d_{f^i(\psub)}(T, f(T)) + O(k).
\]
Hence there is still a hierarchy-like sum of the form 
\[
d_{\psub}(\lambda^+, \lambda^-)  \asymp \sum_{i \in \ZZ} \left[ d_{f^i(\psub)}(T, f(T)) \right]_{}.
\]

\end{remark}

\section{Distance formulas}\label{Sec:DistanceFormulas}
As part of his 
work on the coarse geometry of the Teichm\"uller metric \cite{Rafi:Combinatorial}, Rafi has proved a hierarchy-like distance formula for Teichm\"uller distance. (See also Durham \cite[Theorem 7.4.7]{Durham:AugmentedMarkingComplex} for an alternate proof.) In this section, we show how Rafi's formula can be used to give a uniform coarse estimate for Teichm\"uller translation length. This result, \Cref{Prop:RafiInterpreted}, appears to be new.

For the statement, recall our conventions for hierarchy-like sums (\Cref{Def:HierarchyLikeSum}).

\begin{proposition}\label{Prop:RafiInterpreted}
 For any pseudo-Anosov $f$ with invariant laminations $\lambda^\pm$, there is a hierarchy-like sum
\[
\ell_{\T(S)}(f) \asymp \ell_{S}(f) + \sum_{\langle f \rangle Y} \big[ d_Y(\lambda^+,\lambda^-) \big] + \sum_{\langle f \rangle A} \big[ \log  d_A(\lambda^+,\lambda^-) \big].
\]
where the first sum is over $f$--orbits of proper non-annular subsurfaces and the second is over $f$--orbits of annuli. 
\end{proposition}

Before proving \Cref{Prop:RafiInterpreted}, we introduce a definition and two lemmas.

\begin{definition}\label{Def:NSeparated}
Let $x \in \AC(S)$. For $n \in \NN$, a proper essential subsurface $Y \subset S$ is called \define{$n$--separated} (with respect to $x$) if no geodesic ray in $\AC(S)$ from $f^{-n}(x)$ to $\lambda^{-}$ or from $f^n(x)$ to $\lambda^{+}$ enters the $1$--neighborhood of $\bdy Y$. Observe that every $Y$ is $n$--separated for sufficiently large $n$.
\end{definition}

\begin{lemma}\label{Lem:NSeparatedDist}
Let $x \in \AC(S)$, and let $Y$ be an $n$--separated subsurface with respect to $x$. Then 
\[
d_Y(\lambda^+,\lambda^-) - 2M \leq d_Y(f^{-n}(x), f^{n}(x)) \leq d_Y(\lambda^+,\lambda^-) + 2M, 
\]
where $M$ is the constant of \Cref{Thm:BGIT}.
\end{lemma}

\begin{proof}
Set $x_n = f^{-n}(x)$ and $y_n = f^n(x)$.
By the bounded geodesic image theorem (\Cref{Thm:BGIT}), any ray from $x_{n}$ to $\lambda^{-}$ (respectively $y_{n}$ to $\lambda^{+}$) projects to a path of diameter at most $M$ in $\mathcal{AC}(Y)$. It follows by the triangle inequality that 
\begin{align*}
d_{Y}(\lambda^{-}, \lambda^{+}) &\leq d_{Y}(\lambda^{-}, x_{n}) + d_{Y}(x_{n}, y_{n}) + d_{Y}(y_{n}, \lambda^{+})\\
&\le d_{Y}(x_{n}, y_{n}) + 2M.
\end{align*}
The other inequality in the lemma is proved in exactly the same way.
\end{proof}

\begin{lemma}\label{Lem:NSeparatedCount}
Let $Y$ be a proper essential subsurface with $d_Y(\lambda^+, \lambda^-) > M$. Then, for all $x \in \AC(S)$ and all sufficiently large $n$ (depending on $x$), there exist $n$ distinct $n$--separated translates $Y_{1},..., Y_{n} \in \langle f \rangle (Y)$.
\end{lemma}

\begin{proof}
Since $d_{Y}(\lambda^{+}, \lambda^{-}) > M$, \Cref{Thm:BGIT} implies that $\bdy Y$ lies in the $2$--neighborhood of every geodesic in $\AC(S)$ connecting $\lambda^{-}$ to $\lambda^{+}$. Let $g$ be one such bi-infinite geodesic. Since $d_{f^iY}(\lambda^{+}, \lambda^{-}) > M$ for every $i$, 
it follows that 
every $\langle f \rangle$--translate of $\bdy Y$
 lies in a $2$--neighborhood of $g$. In particular, for every $i,j$, the distance $d_S(f^i(\bdy Y), f^{i+j}(\bdy Y))$ is within uniformly bounded additive error, say $\kappa_{1}$, from $|j| \cdot \ell_{S}(f)$.

As above, set $x_n = f^{-n}(x)$ and $y_n = f^n(x)$. For every $n$, let $p_n$ be a vertex of $g$ closest  to $x_n$, and let $q_n$ be a vertex of $g$ closest to $y_n$. Since all $x_n, y_n$ lie in a uniformly bounded neighborhood of $g$, it follows that there is a uniform $\kappa_2$ such that
\begin{align}
2n \ell_S(f) - \kappa_2 & \leq d_S(p_n, q_n)  \leq 2n \ell_S(f) + \kappa_2, \label{distpq} \\
2n \ell_S(f) - \kappa_2 & \leq d_S(x_n, y_n)  \leq 2n \ell_S(f) + \kappa_2. \label{distxy}
\end{align}
Combining \eqref{distpq} with the above estimate on $d_S(f^i(\bdy Y), f^{i+2n}(\bdy Y))$, we learn that the number of $f$--translates of $\bdy Y$ whose closest points on $g$ lie between $p_n$ and $q_n$ is within  uniform additive error of $2n$. If $f^j(\bdy Y)$ is such a translate whose closest point on $g$ is more than $\delta+1$ past $p_n$, then no geodesic ray from $x_n$ to $\lambda^-$ can enter the $1$--neighborhood of $f^j(\bdy Y)$. An analogous statement holds near $q_n$. We conclude that for large $n$, there are at least $n$ distinct $n$--separated subsurfaces in the $f$--orbit of $Y$.
\end{proof}

\begin{proof}[Proof of \Cref{Prop:RafiInterpreted}]
We will make use of Rafi's combinatorial model for the Teichm{\"u}ller metric \cite[Theorem 6.1]{Rafi:Combinatorial} and refer the reader to his paper for additional details. A convenient restatement of Rafi's distance estimate
 \cite[Theorem 2.4 and Remark 2.5]{Rafi:Hyperbolicity} says that for every $x,y \in \T(S)$, 
\begin{align}
d_{\teich}(x,y) & \asymp \sum_{Y} \big[ d_{Y}(\mu_{x}, \mu_{y}) \big] + \sum_{\gamma \notin \mu_{x}^{\circ} \cup \mu_{y}^{\circ}} \big[ \log d_{A_\gamma}(\mu_{x}, \mu_{y}) \big] 
\label{subsurface} \\
& + \sum_{\alpha \in \mu_{x}^{\circ} \setminus \mu_{y}^{\circ}} \log \frac{1}{\ell_{x}(\alpha)} + \sum_{\beta \in \mu_{y}^{\circ} \setminus \mu_{x}^{\circ}} \log \frac{1}{\ell_y(\beta)}
\label{annular lengths}  \\
& + \sum_{\gamma \in \mu_{x}^{\circ} \cap \mu_{y}^{\circ}} d_{\mathbb{H}}\left( \Big( \mbox{twist}_{\gamma}(x,y), \frac{1}{\ell_{x}(\gamma)} \Big), \: \Big( 0, \frac{1}{\ell_y(\gamma)} \Big) \right).
\label{hyperbolic plane terms} 
\end{align}
In the above expression, $\mu_z$ is a carefully chosen ``short marking'' on the Riemann surface $z\in \T(S)$, and $\mr \mu_z$ is the set of base curves (not transversals) in the marking $\mu_z$. This marking is chosen in a $\Mod(S)$--equivariant manner. 
The first sum in \eqref{subsurface} runs over all non-annular sub-surfaces $Y \subset S$, including $S$ itself. 
In the second sum of \eqref{subsurface}, $A_\gamma$ is the annulus with core curve $\gamma$. Finally, in  \eqref{annular lengths} and \eqref{hyperbolic plane terms}, $\ell_z(\gamma)$ denotes the length of $\gamma$ in the hyperbolic structure $z \in \T(S)$. 
While the expressions in the last two lines are quite complicated, we will show that they do not contribute to the estimate of translation length $ \ell_{\teich(S)}(f)$.

Now, let $x$ be an arbitrary point on the Teichm{\"u}ller axis for $f$, and set $x_{n} = f^{-n}(x)$  and $y_{n} = f^{n}(x)$ for $n \in \mathbb{N}$.
Then we have 
\begin{equation} \label{asymptotic trans}
 \ell_{\teich(S)}(f) = \lim_{n \rightarrow \infty}\frac{1}{2n} \cdot d_{\teich} (x_{n}, y_{n}). 
\end{equation}
Moreover, Rafi's combinatorial formula gives an expression for the right hand side of \eqref{asymptotic trans} in terms of subsurface projections, lengths, and twist parameters. Indeed, denoting the sums in \eqref{subsurface}, \eqref{annular lengths}, and \eqref{hyperbolic plane terms} by $A(x,y), B(x,y)$, and $C(x,y)$ respectively, one has 
\begin{equation}\label{RafiABC}
 \lim_{n \rightarrow \infty}\frac{1}{2n} \cdot d_{\teich} (x_{n}, y_{n}) 
\asymp \lim_{n\rightarrow \infty} \frac{A(x_{n}, y_{n}) + B(x_{n}, y_{n})+ C(x_{n}, y_{n})}{n}
\end{equation}

We next claim that $B(x_{n}, y_{n})$ and $C(x_{n}, y_{n})$ are both bounded above uniformly in $n$. Indeed, 
the uniform upper bound on $B(x_{n}, y_{n})$ holds because
all points $x_{n}$ and $y_{n}$ lie in the same $\Mod(S)$ orbit, hence there is a uniform lower bound on injectivity radius over all such points.
 As for $C(x_{n}, y_{n})$, we simply note that for $n$ sufficiently large, the sum is empty because $f$ is pseudo-Anosov and so for instance, it follows that the distance in the curve complex between $\mu^{\circ}_{x_{n}}$ and $\mu^{\circ}_{y_{n}}$ goes to infinity. In particular, these pants decompositions share no curves in common. 
Hence, \eqref{asymptotic trans} and \eqref{RafiABC} simplify to
\begin{equation} \label{A to T}
\ell_{\teich}(f) = \lim_{n \rightarrow \infty}\frac{1}{2n} \cdot d_{\teich} (x_{n}, y_{n}) 
\asymp \lim_{n\rightarrow \infty} \frac{A(x_{n}, y_{n})}{n}. 
\end{equation}

\smallskip

In what follows, we fix a short marking $\mu = \mu_x$, and let $\mu_{x_n} = f^{-n}(\mu)$ and $\mu_{y_n} = f^n(\mu)$ be short markings at $x_n$ and $y_n$, respectively. 
For any subsurface $Y$, we use the shorthand notation $\pi_Y(x_n)$ to mean $\pi_Y(\mu_{x_n}) = \pi_Y(f^{-n}(\mu))$, and similarly for $y_n$. In particular, $d_Y(x_n, y_n)$ denotes the diameter in $\AC(Y)$ of the projections of $f^{-n}(\mu) \cup f^n(\mu)$. Similarly, $d_S(x_n, y_n)$ denotes the distance in $\AC(S)$ between the short markings at $x_n$ and $y_n$.

Recall from \Cref{Def:NSeparated} that every proper essential subsurface $Y$ is $n$--separated (with respect to the marking $\mu = \mu_x$) for sufficiently large $n$.
Thus, by \Cref{Lem:NSeparatedDist},  if $d_{Y}(\lambda^{+}, \lambda^{-})$ is sufficiently large, then  $Y$ eventually corresponds to a summand of $A(x_{n}, y_{n})$. By \Cref{Lem:NSeparatedCount}, there are $n$ distinct translates of $Y$ by powers of $f$ that correspond to summands of $A(x_n, y_n)$.

Combining the two lemmas, we conclude that for each proper non-annular subsurface $Y$ such that $d_Y(\lambda^{-}, \lambda^{+})$ is sufficiently large, there are at least $n$ distinct summands of $A(x_{n}, y_{n})$ that are all within $2M$ of $d_{Y}(\lambda^{-}, \lambda^{+})$. For each annular subsurface $A$ such that $d_A(\lambda^{-}, \lambda^{+})$ is sufficiently large, there are at least $n$ distinct summands of $A(x_{n}, y_{n})$ that are all within a bounded distance of $\log d_{A}(\lambda^{-}, \lambda^{+})$. Finally, by \eqref{distxy}, we also have  $d_S(x_{n}, y_{n}) \geq 2n \cdot \ell_{S}(f) - \kappa_{2}$ in the full surface $S$.
Putting it all together gives
\[ A(x_{n}, y_{n})  \geq  
2n \cdot \ell_{S}(f) - \kappa_{2}  +   \sum_{\langle f \rangle Y} \Big( n \,  \big[ d_Y(\lambda^+,\lambda^-) \big]_{}  \Big) +  \sum_{\langle f \rangle A} \Big( n \,  \big[ \log d_A(\lambda^+,\lambda^-) \big]_{} \Big).  \]
Note that the sums on the right hand side range over $f$-orbits of subsurfaces.   
Now, dividing by $n$, letting $n \to \infty$, and applying \eqref{A to T} establishes the lower bound on $\ell_{\teich}(f)$ in Proposition \ref{Prop:RafiInterpreted}.

For the upper bound on $\ell_{\teich}(f)$, suppose $Y \subset S$ is a proper subsurface such that $d_{Y}(x_{n}, y_{n})$ is large enough to survive the cutoff in $A(x,y)$. In particular, we may assume that  $d_{Y}(x_{n}, y_{n})> M$.
By the bounded geodesic image theorem, \Cref{Thm:BGIT}, $\bdy Y$ must intersect the $1$-neighborhood of a given geodesic in $\AC(S)$ from $x_{n}$ to $y_{n}$. By an identical argument, any translate $f^i(Y)$ satisfying $d_{f^i Y}(x_{n}, y_{n})> M$ must also have its boundary intersecting the $1$-neighborhood of the same geodesic from $x_{n}$ to $y_{n}$. Thus, by the argument in the paragraph surrounding \eqref{distxy}, the number of such translates is coarsely bounded by $n$.

 In particular, the number of $\langle f \rangle$--translates of $Y$ that make a nonzero contribution to $A(x,y)$ is coarsely bounded by $n$.

To estimate the summand $[d_{Y}(x_{n}, y_{n})]$, we use 
a theorem of Rafi~\cite[Theorem 6.1]{Rafi:Hyperbolicity}
which states that the projection of a Teichm{\"u}ller geodesic to $\mathcal{AC}(Y)$ does not backtrack, and in particular, the diameter of a projection of a sub-geodesic can be at most the diameter of the projection of an entire geodesic, up to some uniform additive error, $J\ge 0$. Since $x_n$ and $y_n$ lie along the Teichmuller geodesic from $\lambda^-$ to $\lambda^+$, Rafi's result gives
\[ d_{Y}(x_{n}, y_{n}) \leq d_{Y}(\lambda^{+}, \lambda^{-}) + J. \]
Combining this fact with \eqref{distxy} gives
\[ A(x_{n}, y_{n}) \prec 2n \cdot  \ell_{S}(f)  +   \sum_{\langle f \rangle Y} \Big( n \,  \big[ d_Y(\lambda^+,\lambda^-) \big]_{}  \Big) +  \sum_{\langle f \rangle A} \Big( n \,  \big[ \log d_A(\lambda^+,\lambda^-) \big]_{} \Big).  \]
The upper bound on $\ell_{\teich}(f)$ follows from dividing both sides  by $n$ and taking a limit as $n \rightarrow \infty$. 
%
\end{proof}


We will also need the following result of Choi and  Rafi \cite[Corollary D]{choi2007comparison}, which is also an application of Rafi's distance formula \cite[Theorem 2.4]{Rafi:Hyperbolicity}.

\begin{proposition}[Choi--Rafi]\label{Prop:ChoiRafi}
For any pair of markings $\mu_1,\mu_2$ on $S$,
\[
\log i(\mu_1,\mu_2) \asymp d_S (\mu_1, \mu_2) + \sum_Y \big[ d_Y(\mu_1,\mu_2) \big] + \sum_A \big[ \log \left (d_A(\mu_1,\mu_2) \right ) \big],
\]
where the first sum is over proper non-annular subsurfaces of $S$, and the second sum is over annuli. 
\end{proposition}

We remark that the statement of \cite[Corollary D]{choi2007comparison} includes specified cutoffs, and does not assert that the sum is hierarchy-like. However, the essential ingredient in the proof of \cite[Corollary D]{choi2007comparison} is Rafi's distance formula  \cite[Theorem 2.4]{Rafi:Hyperbolicity}, and that distance formula is hierarchy-like by  \cite[Remark 2.5]{Rafi:Hyperbolicity}.

\section{The proof of \Cref{th:intro_main}}\label{Sec:StronglyIrreducibleProof}
In this section, we prove \Cref{th:intro_main}: for a strongly irreducible pseudo-Anosov homeomorphism  $f\colon S \to S$,
\[
\ell_{\teich(S)}(f)  \asymp \log \#\Fix(f),
\]
where $\ell_{\T(S)}(f) = \log \lambda$ is translation length in the Teichm\"uller metric.
This coarse equality follows
by chaining together three top-level coarse comparisons, namely \Cref{Prop:TeichDistTriangIntersection,Prop:FromTriangToEdge,Prop:EdgeToFixedPoints} below.

\begin{proposition}\label{Prop:TeichDistTriangIntersection}
Let $f$ be a pseudo-Anosov and 
let $T$ be an annular-avoiding section. Then 
\[
\log i(T, f(T)) \asymp \ell_{\T(S)}(f).
\]
\end{proposition}

\begin{proof}
We compute as follows:
\begin{align}
\log i(T, f(T))& \asymp d_{\mr S}(T, f(T)) + \sum_{\mr Y} \, \big[d_{\mr Y}(T, f(T))\big] \: +\sum_{A} \, \big[\log  d_{A}(T, f(T)) \big]
\label{sumoverTs} \\
& \asymp \ell_{\mr S}(\mr f) \qquad \:\: + \sum_{\langle f \rangle \mr Y}\big[d_{\mr Y}(\lambda^+,\lambda^-)\big] \:\: + \sum_{\langle f \rangle A} \big[\log d_A(\lambda^+,\lambda^-)\big] 
\label{sumoverlambdas} \\
& \asymp \ell_{\teich(\mr S)}(\mr f) 
\label{puncturedteichdist} \\
& = \ell_{\teich(S)}(f).
\end{align}

The first coarse equality is \Cref{Prop:ChoiRafi} applied to the punctured surface $\mr S$. Here, the triangulation $T$ is treated as a marking by arcs on $\mr S$. 
The coarse equality in \eqref{sumoverlambdas} uses 
the hypothesis that that $T$ is annular-avoiding
to match up each summand of \eqref{sumoverTs} with the corresponding summand of \eqref{sumoverlambdas} using the lemmas of \Cref{sec:aas}. Indeed, the terms corresponding to $\mr S$ are coarsely equal by \Cref{Lem:SectionsAlongGeodesic}; the summands corresponding to non-annular proper subsurfaces are coarsely equal by \Cref{Lem:fSectionSum}; and the summands corresponding to annuli are coarsely equal by \Cref{Lem:AnnularAvoiding}. After replacing each summand by one that is coarsely equal, the whole sum is coarsely preserved by  \Cref{Lem:HierarchyLikeSum}.
The coarse equality in \eqref{puncturedteichdist} is \Cref{Prop:RafiInterpreted}, applied to $\mr f$, and the last equality is the standard fact that $f$ and $\mr f$ have the same Teichm\"uller translation length.
\end{proof}

For the next proposition, which is the most important and novel portion of the proof of \Cref{th:intro_main}, we need a definition and a lemma.

\begin{definition}\label{Def:SSum}
 Let $x_1, x_2$ be one-dimensional objects in $\mr S$. To ease notation, define the hierarchy--like sum (compare \Cref{Def:HierarchyLikeSum})
\[
\mc S (x_1,x_2) = d_{\mr S}(x_1,x_2) + \sum_{\mr Y} \, \big[ d_{\mr Y}(x_1,x_2) \big] +\sum_{A} \, \big[ \log  d_{A}(x_1,x_2) \big],
\]
where the first sum is over essential, proper non-annular subsurfaces of $\mr S$ and the second is over annuli. 
Recall from \Cref{Sec:Projection} that projection distances are defined as a diameter, and the diameter of the empty set is $0$ by convention. 
We observe that the proof of 
\Cref{Prop:TeichDistTriangIntersection} already involves the sum $\mc S(T, f(T))$.
\end{definition}

\begin{lemma}\label{Lem:SSumSigma}
Let $f$ be a pseudo-Anosov, and let $\sigma$ be any edge of the veering triangulation. Then
\[
\log i(\sigma, f(\sigma)) \asymp \mc S(\sigma, f(\sigma)),
\]
with the convention that $\log(0)=0$.
\end{lemma}

\begin{proof}
Observe that if $i(\sigma, f(\sigma))=0$, then  $\mc S(\sigma, f(\sigma)) \asymp 0$ as well, and the result follows by convention.

Now, suppose that $i(\sigma, f(\sigma)) > 0$. Then $\sigma > f(\sigma)$ since our convention is that $f$ stretches horizontally. Let $P = P(\sigma, f(\sigma))$ be the unique pocket defined by $\sigma$ and $f(\sigma)$, as in \Cref{Def:PartialSections}. We extend the top and bottom partial sections, namely $P(\sigma, f(\sigma))^+$ and $P(\sigma, f(\sigma))^-$, to full sections $T(\sigma, f(\sigma))^+$ and $T(\sigma, f(\sigma))^-$ that agree away from the pocket $P$. Thus any intersections between the two sections must occur in the base of the pocket.
Now,
we compute as follows:
\begin{align*}
\log i(\sigma, f(\sigma)) &\asymp \log i(P(\sigma, f(\sigma)^+ \!, \, P(\sigma, f(\sigma))^-) \\
&\asymp \mc S (P(\sigma, f(\sigma)^+ \!, \, P(\sigma, f(\sigma))^-) \\
& \asymp  \mc S (\sigma, f(\sigma)). 
\end{align*}
The first coarse equality is \Cref{Lem:ClosestPartialSections}. The second coarse is \Cref{Prop:ChoiRafi}, where we are treating the ideal triangulations $T(\sigma, f(\sigma))^+$ and $T(\sigma, f(\sigma))^-$  as markings; recall that all intersections (hence, all subsurface distances) come from the top and bottom of the pocket $P$.
 The third coarse equality follows from \Cref{Lem:ClosestSecSubsurfaceDist}, together with \Cref{Lem:HierarchyLikeSum}.
\end{proof}

We can now prove the following proposition, which is the one crucial step of the proof that requires $f$ to be strongly irreducible.

\begin{proposition}\label{Prop:FromTriangToEdge}
Let $f$ be strongly irreducible, and let $T$ be a section. 
Let $\sigma$ be an edge of $T$ for which the intersection number $i(\sigma, f(\sigma))$ is largest. Then
\[
\log i(\sigma, f(\sigma)) \asymp \log i(T, f(T)),
\]
with the convention that $\log(0)=0$.
\end{proposition}

\begin{proof}
Let $\sigma'$ be an arbitrary edge of $T$. By \Cref{Lem:SSumSigma}, combined with the hypothesis that $\sigma$ is the edge of $T$ that maximizes the intersection number $i(\sigma, f(\sigma))$, we have
\[
\mc S(\sigma', f(\sigma'))  \asymp  \log i(\sigma', f(\sigma')) \leq 
 \log i(\sigma, f(\sigma)) \asymp  \mc S(\sigma, f(\sigma)).
\]
Comparing the first and last terms, we obtain
\begin{equation}\label{Eqn:SigmaLargest}
\mc S(\sigma', f(\sigma')) \prec \mc S(\sigma, f(\sigma))
\end{equation}
for every $\sigma' \subset T$.

Next, we relate $\mc S (\sigma, f(\sigma))$ to $\mc S(T, f(T))$ by a partitioning argument, as follows. 
Since $f$ is strongly irreducible, every subsurface $\mr Y$ is overlapped by $f$.
Thus, by  \Cref{Lem:CuttingSubsurfaces}, 
every subsurface $\mr Y$ that contributes to the sum $\mc S(T, f(T))$ must be cut by both $\sigma'$ and $f(\sigma')$ for \emph{some} edge $\sigma'$ of $T$, depending on $\mr Y$. Hence, we have $d_{\mr Y}(T, f(T)) = d_{\mr Y}(\sigma', f(\sigma')) + O(1)$ for 
this choice of $\sigma'$.
Then we obtain
\[
\mc S(\sigma, f(\sigma)) \leq \mc S(T, f(T)) \prec \sum_{\sigma' \subset T} \mc S(\sigma', f(\sigma'))
\prec \mc S(\sigma, f(\sigma)).
\]
Here, the first inequality is just the fact that $\sigma \subset T$. The next (coarse) inequality comes from decomposing $\mc S(T, f(T))$ into a sum over edges: for each edge $\sigma'$, let $\mc S(\sigma', f(\sigma'))$ be the sub-sum of $\mc S(T, f(T))$ such that $\sigma'$ and $f(\sigma')$ cut the subsurface in question. Since each subsurface with an associated summand in $\mc S(T, f(T))$ contributes to a summand in at least one such $\mc S(\sigma', f(\sigma'))$, the second inequality follows. The final coarse inequality uses the fact that there are at most $6|\chi(S)|$ edges of $T$, 
combined with \eqref{Eqn:SigmaLargest}.

Since the first and last terms of the above equation are equal, we obtain
\[
\mc S(\sigma, f(\sigma)) \asymp \mc S(T, f(T)).
\]
Putting it all together gives
\[
\log i(\sigma, f(\sigma)) \asymp \mc S(\sigma, f(\sigma)) \asymp \mc S(T, f(T)) \asymp \log i(T, f(T)), 
\]
where the last coarse equality is again \Cref{Prop:ChoiRafi}.
\end{proof}

%
%
%
\begin{proposition}\label{Prop:EdgeToFixedPoints}
Let $T$ be an annular avoiding $f$--section. Let $\sigma$ be an edge of $T$ for which the intersection number $i(\sigma, f(\sigma))$ is largest. Then
\[
 \log \# \Fix(f) = \log  i(\sigma, f(\sigma)) + O(\log |\chi(S)|).
\]
\end{proposition}

\begin{proof}
For one direction of the estimate, observe that every nonsingular point 
of $S$ belongs to an immersed rectangle $R_{\sigma'}$ corresponding to some edge $\sigma' \subset T$. By the fundamental lemma (\Cref{Lem:FundamentalLemma}), the number of fixed points in $R_{\sigma'}$ is at most $i(\sigma', f(\sigma'))$. Furthermore, the number of singular points is at most $2| \chi(S) |$; compare \Cref{Lem:ClosestPartialSections}.
Thus
\[
\# \Fix(f) \leq \sum_{\sigma' \subset T} i(\sigma', f(\sigma')) + 2| \chi(S) |
\leq | 9 \chi(S) |  i(\sigma, f(\sigma)) + 2| \chi(S) |. 
\]
The second inequality holds because there are at most $| 9 \chi(S) |$ edges in $T$, and $\sigma$ was  chosen to maximize intersection numbers.

Now, consider the other direction. 
Since $T$ is annular-avoiding, \Cref{Lem:AnnularAvoiding} and \Cref{Lem:LowDegreeAwayFromPocket} imply that the immersed spanning rectangle for $\sigma$ has degree uniformly bounded by $13$.
Thus, by the fundamental lemma, 
\[
 \# \Fix(f) \geq \frac{1}{13}  i(\sigma, f(\sigma)).
 \qedhere
\]
\end{proof}

We can now complete the proof of  \Cref{th:intro_main}.

\begin{proof}[Proof of \Cref{th:intro_main}]
Let $T$ be an annular-avoiding $f$--section, and let $\sigma \subset T$ be be an edge of $T$ for which the intersection number $i(\sigma, f(\sigma))$ is largest. 
Chaining together the coarse equalities of \Cref{Prop:TeichDistTriangIntersection,Prop:FromTriangToEdge,Prop:EdgeToFixedPoints} gives
\[
\log \# \Fix(f) \asymp \log i(\sigma, f(\sigma))  \asymp  \log i(T, f(T))  \asymp \ell_{\teich(S)}(f).
\qedhere
\]
%
\end{proof}

\subsection{Remarks on the argument}
Next, we make two remarks on the proof of  \Cref{th:intro_main}.

\begin{remark}\label{Rem:BacktrackingProof}
The above proof of \Cref{th:intro_main} is slightly inefficient. Our line of argument applies  \Cref{Prop:ChoiRafi} to go from $\log i(T, f(T))$ to $\mc S(T, f(T))$ in the proof of \Cref{Prop:TeichDistTriangIntersection}, and then applies   \Cref{Prop:ChoiRafi} again to go from $\mc S(T, f(T))$ back to $\log i(T, f(T))$ in the proof of \Cref{Prop:FromTriangToEdge}. It would be more efficient to carry the quantity $\mc S(T, f(T))$ all through the proof. However, this would come at the cost of making the statements of \Cref{Prop:TeichDistTriangIntersection,Prop:FromTriangToEdge} potentially harder to read and absorb.
\end{remark}

\begin{remark}\label{Rem:UpperBoundMarkov}
There is a straightforward upper bound on $\log \# \Fix(f)$ in terms of $\ell_{\teich(S)}(f)$ that does not use the methods developed in this paper. To see this, take a Markov partition for $f\colon S \to S$, with $n$ rectangles.  Using the method in the Bestvina--Handel algorithm \cite{besthand}, we may take $n$  bounded by a linear function of $|\chi(S)|$. Let $A$ be the $n \times n$ transition matrix for the partition. Then the Perron--Frobenius eigenvalue of $A$ is the stretch factor $\lambda$, where $ \ell_{\teich(S)}(f) = \log \lambda$. Since $\#\Fix(f)$ is within an additive error of $\operatorname{trace}(A)$, and $\operatorname{trace}(A) \leq n \lambda$, we obtain
\[
\#\Fix(f) \le c \cdot \lambda(f) + c,
\]
where $c$ is linear in $\chi(S)$. It follows that 
\[
\log \#\Fix(f) \leq \log \lambda(f) + O(\log |\chi(S)|) = \ell_{\teich(S)}(f) + O(\log |\chi(S)|).
\]
This estimate is stronger than the upper bound of \Cref{th:intro_main}, because it does not involve any multiplicative error.
\end{remark}

\section{Without strong irreducibility}\label{Sec:WithoutStronglyIrreducible}

When $f$ is not strongly irreducible, it is no longer true that the log of the number of fixed points is comparable to Teichm{\"u}ller translation length. Indeed, explicit counterexamples are produced 
in \Cref{Const:twist}.

Recall from the introduction that $f$ \define{overlaps} an essential subsurface $Y \subset S$ if $f(Y)$ intersects $Y$ essentially. 
Thus, strong irreducibility is equivalent to every subsurface (including $S$ itself) being overlapped. In the absence of strong irreducibility, one might still hope to relate the log of the number of fixed points to the action of the map on the subsurfaces which \emph{are} overlapped. For this, let $\mc D_f$ be the set of (isotopy classes of) subsurfaces that $f$ overlaps:
\[
\mc D_f = \{Y : f \text{ overlaps the subsurface } Y \} .
\]

\begin{theorem}[Fixed points; general case] \label{thm:NSI}
Let $f \colon S \to S$ be a pseudo-Anosov. Then 
\[
\log \# \Fix(f) \asymp  \ell_{S}(f) + \sum_{\langle f \rangle Y \subset \mc D_f \setminus \{ S \}} \big[ d_Y(\lambda^+,\lambda^-) \big] + \sum_{\langle f \rangle A \subset \mc D_f} \big[ \log d_A(\lambda^+,\lambda^-) \big],
\]
where the first sum is over $f$--orbits of proper non-annular subsurfaces and the second is over $f$--orbits of annuli. In each case, we take exactly one subsurface per orbit. 
\end{theorem}

\begin{remark}
We emphasize that this theorem does not follow from the arguments in previous sections. Indeed, 
\Cref{Prop:FromTriangToEdge} uses strong irreducibility in a crucial way to 
 deduce a comparison between $i(T, f(T))$ for a section $T$ to   $i(\sigma, f(\sigma))$ for a single edge $\sigma \subset T$. The proof of this proposition requires a careful consideration of all the subsurfaces at once. In what follows, we will argue more locally, relating the action of the map $f$ on a particular overlapped subsurface $Y \subset S$ to the action of $f$ on all subsurfaces of the \emph{punctured} surface $\mr S$ that fill to $Y$. 

Going in the other direction, the combination of \Cref{thm:NSI} and \Cref{Prop:RafiInterpreted} immediately implies  \Cref{th:intro_main}, because every subsurface of $S$ is overlapped when $f$ is strongly irreducible.
However, the arguments that prove \Cref{thm:NSI}  are far more technical than the proofs in previous sections.
In addition, the arguments that will follow also rely on many of the lemmas in the preceding sections.
\end{remark}

The proof uses the following setup. For any essential subsurface $\psub \subset \mr S$, recall from  \Cref{Sec:FillingSubsurfaces} that we defined $\Fill_S(\psub)$ be the subsurface $Y \subset S$ obtained by capping off any components of $\bdy \psub$ that are inessential in $S$, as well as any punctures of $\psub$ that are filled in $S$. Now, given an essential subsurface $Y \subset S$, we define $\mc F_Y$ to be the collection of (isotopy classes of) subsurfaces of $\mr S$ that fill to $Y$:
\[
 \mc F_Y = \{ \psub \subset \mr S : \Fill_S(\psub) = Y \} / \text{(isotopy)}.
\]

A priori, the set $\mc F_Y$ may be difficult to control: for example, it contains many subsurfaces created by complicated point-pushes around the singularities of $S$ contained in and outside of $Y$. 
However, we primarily care about subsurfaces $\psub \in \mc F_Y$ such that $d_{\psub}(\lambda^-, \lambda^+)$ is large. By \Cref{Lem:CompatibleProj}, any such subsurface $\psub$ is $q$-compatible and has a canonical $q$--geodesic representative. 
In fact, when $Y$ itself is $q$--compatible (the only case of interest here), we can define a canonical element of $\mc F_Y$ using the $q$--geometry:
\[
\mr Y = \int_q Y \setminus \sing(q) \subset \mr S.
\]

The following proposition says that distances in $Y$ can be estimated in terms of  precisely the elements of  $\mc F_Y$ for which $d_{\psub}(\lambda^-, \lambda^+)$ is sufficiently large.

\begin{proposition}[Distance shuffling]
\label{Prop:Shuffling}
Puncturing $S$ to produce $\mr S$ affects subsurface distances as follows:
\begin{enumerate}
\item \label{Itm:FullSurfShuffling}
For the full surface $S$, we have
\[
\ell_{S}(f) \asymp \ell_{\mr S}(\mr f) + \sum_{\langle f \rangle \psub \subset \mc F_S \setminus \{\mr S\}} \big[ d_{\psub}(\lambda^+,\lambda^-) \big],
\]
where the sum is over $f$--orbits of proper subsurfaces that fill to $S$.
\item \label{Itm:ProperSubsurfShuffling}
For any proper non-annular subsurface $Y \subsetneq S$, we have
\[
d_Y(\lambda^+, \lambda^-) \asymp \sum_{\psub \in \mc F_Y} \big[ d_{\psub}(\lambda^+,\lambda^-) \big].
\]
\item \label{Itm:AnnulusShuffling}
For any essential annulus $A \subset S$, any subsurface  $\psub \in \mc F_A$ with $d_{\psub}(\lambda^+,\lambda^-) \geq 4$ must be a single annulus, satisfying
\[
d_A(\lambda^+, \lambda^-) = d_{\psub}(\lambda^+, \lambda^-).
\]
\end{enumerate}
\end{proposition}

 \Cref{Prop:Shuffling} (distance shuffling) will be proved in \Cref{Sec:ShufflingProof}.

\subsection{Applying distance shuffling}
For now, we will apply \Cref{Prop:Shuffling}  to prove  \Cref{thm:NSI}. 
To begin the argument, recall that $\mc D_f$ is the set of 
subsurfaces that are overlapped by $f$. 
In parallel with this, define $\mr {\mc D}_f$ to be the set of essential subsurfaces $\psub \subset \mr S$ 
that are overlapped by $\mr f$.
The following lemma shows that (under a mild distance hypothesis) the subsurfaces in $\mr {\mc D}_f$ are exactly the ones that fill to elements of $\mc D_f$.

\begin{lemma}\label{Lem:FillingOverlap}
The sets $\mc D_f$ and $\mr {\mc D}_f$ are related as follows:
\begin{itemize}
\item If $Y \in \mc D_f$ and $\psub \in \mc F_Y$, then $\psub \in \mr {\mc D}_f$.
\item If $\psub \in \mr {\mc D}_f$ and  $d_{\psub}(\lambda^+,\lambda^-) \geq 4$ (or $\psub = \mr S$), then $Y = \Fill_S(\psub)  \in \mc D_f $.
\end{itemize}
\end{lemma}

\begin{proof}
For the first bullet,  if $\mr W$ is not overlapped, then $\psub$ and $f(\psub)$ are disjoint up to isotopy in $\mr S$. Filling in punctures then gives disjoint representatives of $Y$ and $f(Y)$, implying that $Y$ is also not overlapped.


For the second bullet, suppose that $\psub \in \mr {\mc D}_f$ and  $d_{\psub}(\lambda^+,\lambda^-) \geq 4$. By \Cref{Lem:EssentialFill}, $Y = \Fill_S(\psub)$ is an essential and $q$-compatible subsurface of $S$; furthermore, there are $q$--convex representatives $\int_q \psub$ and $\int_q Y$ such that $\int_q \psub \subset \int_q Y$. Since $\psub$ is overlapped by $\mr f$, it follows that $f(\int_q Y)$ intersects $\int_q(Y)$, and the intersection is essential by $q$--convexity.
\end{proof}

Recall that the proof of \Cref{th:intro_main} used a number of hierarchy-like sums of the form $\mc S(x_1, x_2)$ (see  \Cref{Def:SSum}). For our argument here, we need to define an analogous sum that only uses subsurfaces in $\mr{\mc D}_f$.

\begin{definition}\label{Def:SSumDf}
Let $x_1, x_2$ be one--dimensional objects in $S$, and let $f$ be a pseudo-Anosov. Define
\[
\mc S_{\mr{\mc D}_f}(x_1, x_2) =  d_{\mr S}(x_1, x_2) + \sum_{\psub \in \mr{\mc D}_f \setminus \{ \mr S \}} \big[ d_{\psub}(x_1, x_2) \big] + \sum_{A \in \mr{\mc D}_f} \big[ \log d_A(x_1, x_2) \big],
\]
where the first sum is over non-annular proper subsurfaces in $\mr{\mc D}_f$ and the second is over annuli in $\mr{\mc D}_f$. 
\end{definition}


The next lemma shows that for any veering edge $\sigma$, a large-projection surface $\psub$ that contributes a summand to $\mc S(\sigma, f(\sigma))$ will also contribute a summand to $S_{\mr{\mc D_f}}(\sigma, f(\sigma))$.

\begin{lemma}\label{Lem:LargeSigmaDistMeansOverlap}
Let $\sigma$ be any edge of the veering triangulation. If $\psub$ is a proper subsurface of $\mr S$ such that $d_{\psub}(\sigma, f(\sigma)) \geq 10$, then $\psub$ is overlapped by $f$.
\end{lemma}


\begin{proof}
Let $T_\sigma$ be a section containing $\sigma$. Since $d_{\psub}(T_\sigma, f(T_\sigma)) \geq 10$, 
 \Cref{Thm:NicePath}\eqref{Itm:VeeringEfficiency} implies  $\psub$ is $\tau$--compatible. Thus $\psub$ has an open representative $\int_\tau \psub$ whose boundary $\bdy_\tau \psub$ consists of veering edges.

If $\sigma$ is contained in $\int_\tau \psub$, the hypothesis $d_{\psub}(\sigma, f(\sigma)) \geq 10$ implies $f(\sigma)$ intersects the interior of $\int_\tau \psub$, hence $f(\int_\tau \psub)$ overlaps $\int_\tau \psub$ as claimed. Similarly, if $f(\sigma)$ is contained in  $\int_\tau \psub$, then $\sigma = f^{-1} \circ f(\sigma)$ intersects the interior of $\int_\tau \psub$, hence $f^{-1}(\int_\tau \psub)$ overlaps $\int_\tau \psub$. It follows that $f$ overlaps $\int_\tau \psub$ as well. 

For the rest of the proof, we will assume that neither $\sigma$ nor $f(\sigma)$ is contained in $\int_\tau \psub$. Since $d_{\psub}(\sigma, f(\sigma)) \geq 10$, both $\sigma$ and $f(\sigma)$ intersect $\int_\tau \psub$, hence both edges must lie above or below $\bdy_\tau \psub$. If both $\sigma$ and $f(\sigma)$ lie above $\bdy_\tau \psub$, then \Cref{lem:above/below} implies both edges are $3$--close to $\lambda^+$, hence  $d_{\psub}(\sigma, f(\sigma)) \leq 6$, a contradiction. The same contradiction arises if both $\sigma$ and $f(\sigma)$ lie below $\bdy_\tau \psub$. Since $f$ translates downward, we conclude that $\sigma$ lies above $\bdy_\tau \psub$ and $f(\sigma)$ lies below $\bdy_\tau \psub$.

Applying $f$, we learn that $f(\sigma)$ lies below $\bdy_\tau \psub$ and above $f(\bdy_\tau \psub)$. Thus some saddle connection in $\bdy_\tau \psub$ must cross the spanning rectangle of $f(\sigma)$ from top to bottom, while some saddle connection in $f(\bdy_\tau \psub)$ must cross the spanning rectangle of $f(\sigma)$ from left to right. It follows that  $\bdy_\tau \psub$ intersects $f(\bdy_\tau \psub)$ essentially, hence $f$ overlaps $\psub$.
\end{proof}

\begin{corollary}\label{Cor:DfFullSum}
For any veering edge $\sigma$, we have $\mc S(\sigma, f(\sigma)) \asymp \mc S_{\mr{\mc D_f}}(\sigma, f(\sigma))$.
\end{corollary}

\begin{proof}
By \Cref{Lem:LargeSigmaDistMeansOverlap}, any subsurface $\psub \subset \mr S$ where $\big[ d_{\psub}(\sigma, f(\sigma)) \big]_{10} > 0$ must actually belong to $\mr{\mc D_f}$, hence $\mc S(\sigma, f(\sigma)) \prec \mc S_{\mr{\mc D_f}}(\sigma, f(\sigma))$. The reverse inequality is immediate, because 
every summand of $\mc S_{\mr{\mc D_f}}(\sigma, f(\sigma))$ appears by definition as a summand of $\mc S(\sigma, f(\sigma))$.
\end{proof}

We can now proceed to the proof of \Cref{thm:NSI}. The proof follows the same outline as the argument in \Cref{Sec:StronglyIrreducibleProof}. In particular, the coarse equality of \Cref{thm:NSI} is proved by chaining together three coarse comparisons: \Cref{Prop:ShufflingTriangIntersection} (which is a replacement for \Cref{Prop:TeichDistTriangIntersection}), \Cref{Prop:FromTriangToEdgeWR} (which is a replacement for \Cref{Prop:FromTriangToEdge}), and \Cref{Prop:EdgeToFixedPoints} (used in its original form).

The next result is the main place where we use \Cref{Prop:Shuffling} (distance shuffling). 

\begin{proposition}\label{Prop:ShufflingTriangIntersection}
Let $f$ be a pseudo-Anosov and let $T$ be an annular-avoiding section. 
Then
\[
\mc S_{\mr{\mc D}_f} (T, f(T)) \asymp  \ell_{S}(f) + \sum_{\langle f \rangle Y \subset {\mc D}_f \setminus \{ S \}} [d_Y(\lambda^+,\lambda^-)] + \sum_{\langle f \rangle A \subset {\mc D}_f} [\log d_A(\lambda^+,\lambda^-)].
\]
\end{proposition}

\begin{proof}
We compute as follows:
\begin{align}
\mc S_{\mr{\mc D}_f} (T, f(T)) &=    d_{\mr S}(T, f(T)) + \sum_{\psub \in \mr{\mc D}_f \setminus \{ \mr S \}} \big[ d_{\psub}(T, f(T)) \big] + \sum_{A \in \mr{\mc D}_f} \big[ \log d_A(T, f(T)) \big] 
\label{sumDf_Ts} \\
& \asymp  \ell_{\mr S}(f) + \sum_{\langle f \rangle \psub \subset \mr{\mc D}_f \setminus \{ \mr S \}} \big[d_{\psub}(\lambda^+, \lambda^-)\big] + \sum_{\langle f \rangle A \subset \mr{\mc D}_f} \big[\log d_A(\lambda^+, \lambda^-)\big] 
\label{sumDf_punctured_lam} \\
& \asymp  \ell_{S}(f) + \sum_{\langle f \rangle Y \subset {\mc D}_f \setminus \{ S \}} \big[d_{Y}(\lambda^+, \lambda^-)\big] + \sum_{\langle f \rangle A \subset {\mc D}_f} \big[\log d_A(\lambda^+, \lambda^-)\big].
\label{sumDF_filled_lam}
\end{align}
Now, we justify the computation. The equality in \eqref{sumDf_Ts} is just \Cref{Def:SSumDf}. The coarse equality in \eqref{sumDf_punctured_lam} is proved in exactly the same way as \eqref{sumoverlambdas}: we use the hypothesis that $T$ is annular-avoiding to match up each summand of \eqref{sumDf_Ts} with the corresponding summand of \eqref{sumDf_punctured_lam}, using \Cref{Lem:AnnularAvoiding,Lem:SectionsAlongGeodesic,Lem:fSectionSum}; after replacing each summand in this manner, the whole sum is coarsely preserved by \Cref{Lem:HierarchyLikeSum}.

Next, we discuss the transition from \eqref{sumDf_punctured_lam} to  \eqref{sumDF_filled_lam}. 
By \Cref{Lem:FillingOverlap}, every subsurface $Y \in \mc D_f$ that contributes a summand to \eqref{sumDF_filled_lam} is the filling of some $\psub \in  \mr{\mc D}_f$, and conversely, every (large-projection) subsurface $\psub \in  \mr{\mc D}_f$ that contributes a summand to  \eqref{sumDf_punctured_lam} fills to some $Y \in \mc D_f$. By \Cref{Prop:Shuffling}, the  summand of \eqref{sumDF_filled_lam} corresponding to $Y$ (or $S$) is coarsely equal to the sum of finitely many terms of \eqref{sumDf_punctured_lam}, corresponding to the subsurfaces $\psub \in \mc F_Y$. 
Thus we may replace each term of \eqref{sumDF_filled_lam} by the corresponding terms of \eqref{sumDf_punctured_lam}, and the coarse equality of the whole sum is again preserved by  \Cref{Lem:HierarchyLikeSum}.
\end{proof}

The next result is a replacement for \Cref{Prop:FromTriangToEdge}, specifically adapted to the scenario where not every subsurface is overlapped.

\begin{proposition}\label{Prop:FromTriangToEdgeWR}
Let $f$ be a pseudo-Anosov and let $T$ be any $f$--section. Let $\sigma$ be an edge of $T$ for which the intersection number $i(\sigma, f(\sigma))$ is largest. Then
\[
\log i(\sigma, f(\sigma)) \asymp \mc S_{\mr {\mc D}_f}(T, f(T)).
\]
\end{proposition}

\begin{proof}
For the edge $\sigma$ as in the statement of the lemma, we have
\[
\log i(\sigma, f(\sigma)) \asymp \mc S(\sigma, f(\sigma)) \asymp \mc S_{\mr{\mc D_f}}(\sigma, f(\sigma)).
\]
Indeed, the first coarse equality is \Cref{Lem:SSumSigma} and the second  is \Cref{Cor:DfFullSum}.

Next, we claim that
\[
\mc S_{\mr{\mc D}_f}(\sigma, f(\sigma)) \asymp \mc S_{\mr {\mc D}_f}(T, f(T)). 
\]
This follows by exactly the same partitioning argument that was used to prove
 \Cref{Prop:FromTriangToEdge}. The key point is that we now restrict ourselves to only those subsurfaces that are overlapped by $f$, so that \Cref{Lem:CuttingSubsurfaces} still applies to every subsurface in the sum. Putting it all together gives
 \[
\log i(\sigma, f(\sigma)) \asymp \mc S(\sigma, f(\sigma)) \asymp \mc S_{\mr{\mc D}_f}(\sigma, f(\sigma)) \asymp \mc S_{\mr {\mc D}_f}(T, f(T)). \qedhere
 \]
\end{proof}

\begin{proof}[Proof of \Cref{thm:NSI}]
Let $T$ be an annular-avoiding $f$--section, and let $\sigma \subset T$ be an edge of $T$ that maximizes $i(\sigma, f(\sigma))$. By \Cref{Prop:EdgeToFixedPoints,Prop:FromTriangToEdgeWR}, we have
\[
\log \# \Fix(f) \asymp \log i(\sigma, f(\sigma))  \asymp \mc S_{\mr {\mc D}_f}(T, f(T)).
\]
Now, by \Cref{Prop:ShufflingTriangIntersection}, the latter sum is coarsely equal to 
\[
\ell_{S}(f) + \sum_{\langle f \rangle Y \subset {\mc D}_f} \big[ d_Y(\lambda^+,\lambda^-) \big] + \sum_{\langle f \rangle A \subset {\mc D}_f} \big[ \log d_A(\lambda^+,\lambda^-) \big],
\]
as desired.
\end{proof}

\subsection{Proving shuffling}\label{Sec:ShufflingProof}

We will prove the statements of \Cref{Prop:Shuffling} in reverse order, starting from the easiest case of annuli.

\begin{proof}[Proof of \Cref{Prop:Shuffling}.\eqref{Itm:AnnulusShuffling}]
Suppose $A \subset S$ is an essential annulus and $\psub  \in \mc F_A$ is a surface filling to $A$, with $d_{\psub}(\lambda^-, \lambda^+) \geq 4$. By  \Cref{Lem:QCompatibleFill}, the interiors  $\int_q \psub$ and $\int_q A$ are non-empty and coincide. In other words, $\int_q \psub$ is a flat annulus whose core curve agrees with that of $\int_q A$, up to isotopy in $\mr S$. Thus $\psub$ is uniquely determined, and
\[
d_A(\lambda^+, \lambda^-) = d_{\psub}(\lambda^+, \lambda^-). \qedhere
\]
\end{proof}

Next, we consider non-annular subsurfaces of $S$.

\subsubsection{Proving \Cref{Prop:Shuffling}.\eqref{Itm:ProperSubsurfShuffling}}
Suppose that $Y \subset S$ is a proper, non-annular essential subsurface. Since we are trying to prove a coarse equality of the form
\[
 d_Y(\lambda^+, \lambda^-) \asymp 
\sum_{\psub \in \mc F_Y} \big[ d_{\psub}(\lambda^+,\lambda^-) \big],
\]
it suffices to consider only those subsurfaces $Y$ and $\psub  \in \mc F_Y$ whose distance between $\lambda^+$ and $\lambda^-$ is sufficiently large. In particular, we may and will assume that $d_{Y}(\lambda^-, \lambda^+) \geq 4$ and $d_{\psub}(\lambda^-, \lambda^+) \geq 4$, which implies that $Y$ and $\psub$ are $q$--compatible by   \Cref{Lem:CompatibleProj}. 
Furthermore, by \Cref{Lem:QCompatibleFill}, we have $\int_q \psub \subset \int_q Y$. Since this implies that $\int_q \psub \subset \int_q \mr Y$, we can naturally consider $\psub $ as a subsurface of $\mr Y \subset \mr S$.

The lower bound on $d_Y(\lambda^+, \lambda^-)$ will need the following definition.

\begin{definition}\label{Def:QProjection}
Let $Y \subset S$ be a $q$--compatible, non-annular subsurface (possibly all of $S$), and let $\psub \in \mc F_Y$ with $d_{\psub}(\lambda^-, \lambda^+) \geq 4$. 
We define a projection  map $\pi^{q}_{\psub} \from \AC(Y) \to \AC(\psub)$, as follows. 

First, we apply the quasi-isometry $c \from \AC(Y) \to \C(Y)$ that is the coarse inverse of the inclusion $\C(Y) \hookrightarrow \AC(Y)$. (For an arc $\beta$,  form the regular neighborhood $N = N(\beta \cup \bdy Y) \subset Y$, let $c(\beta)$ be the union of essential curves in $\bdy N$.)
Then, for an essential curve $\alpha \in \C(Y)$, we apply the map 
$\mathbf{s} \from \C(Y) \to \A(q)$  of  \Cref{Sec:Compatibility} that sends $\alpha$ to the set of (non-crossing) saddle connections that belong to $q$--geodesic representatives of $\alpha$.
 These saddle connections must be contained in $\int_q Y$, because $Y$ is $q$--compatible, so in fact we get $\mathbf{s} \circ c \from \AC(Y) \to \AC(\mr Y)$. Finally, we apply the usual subsurface projection $\pi_{\psub} \from \AC(\mr Y) \to \AC(\psub)$. The combined map 
 \[
\pi^{q}_{\psub} = \pi_{\psub} \circ  \mathbf{s} \circ c \from \AC(Y) \to \A(\psub)
\]
 is the composition of one quasi-isometry and two coarse $1$--Lipschitz retractions, hence $\pi^{q}_{\psub}$ is itself a coarse Lipschitz retraction. 
\end{definition}

Since $\Fill_S(\psub) = Y$, no essential curve or arc in $Y$ can be homotoped off $\psub$. In particular, we have:

\begin{claim} \label{claim:nonempty}
Suppose that $Y \subset S$ is $q$--compatible and that $\psub \in \mc F_Y$ satisfies  $d_{\psub}(\lambda^-, \lambda^+) \geq 4$. Then for each $\alpha$ in $\mc {AC}(Y)$, we have $\pi^q_{\psub}(\alpha) \neq \emptyset$.
\end{claim}

The next lemma follows from well-known arguments that can be applied to any subsurface-like projection.
See Aougab--Taylor--Webb \cite[Proposition 4.1]{AougabTaylorWebb}, Patel--Taylor \cite[Proposition 3]{PatelTaylor:EndperiodicLox}, and most generally,  Sisto--Taylor \cite[Proposition 5.1]{SistoTaylor}

\begin{lemma}\label{Lem:ShufflingUpper}
For any $q$--compatible, non-annular subsurface $Y \subset S$ (with $Y = S$ possible) and any $\alpha, \beta \in \AC(Y)$, we have
\[
d_Y(\alpha, \beta) \succ \sum_{\psub \in \mc F_Y} \left[ d_{\psub} \Big( \pi^{q}_{\psub} (\alpha), \pi^{q}_{\psub}(\beta) \Big) \right] .
\]
\end{lemma}

\begin{proof}
We may regard each of $\alpha$ and $\beta$ as part of a marking on $Y$. Now, 
follow the proof of \cite[Proposition 4.1]{AougabTaylorWebb}, which applies verbatim if one substitutes $\pi_{\psub}^q$ in place of the usual subsurface projection. It is in this proof that the conclusion of Claim \ref{claim:nonempty} is used.
The point is that if $\mu$ is any marking of $Y$, then $\nu = {\bf s}(c (\mu))$ is a marking of $\mr Y$ and $\pi_{\psub}^q (\mu) =  \pi_{\psub}(\nu)$. Hence, the usual properties of subsurface projection apply also to $\pi_{\psub}^q$.
\end{proof}

Now, the lower bound on $d_Y(\lambda^+, \lambda^-)$ in \Cref{Prop:Shuffling}.\eqref{Itm:ProperSubsurfShuffling} will follow by taking $\alpha = \pi_Y(\lambda^-)$ and $\beta = \pi_Y(\lambda^+)$. For the upper bound on $d_Y(\lambda^+, \lambda^-)$, we need to consider the \define{surviving arc and curve graph}, introduced by G\"ultepe, Leininger, and Pho-On \cite{GultepeLeiningerPhoOn}. 

\begin{definition}
Let $Y$ be a $q$--compatible, non-annular subsurface of $S$. Let $\int_q Y$ be the $q$--convex hull of $Y$. Define $\mr Y = \int_q Y \setminus \sing(q)$.

The \define{surviving arc and curve graph} of $\mr Y$, denoted $\AC^s(\mr Y)$, is the graph whose vertices are isotopy classes of essential simple closed curves in $\mr Y$ that remain essential after filling the singularities to recover  in $Y$, as well as essential arcs in $\mr Y$ (with endpoints at punctures of $S$ or boundary components of $Y$) that remain essential in $Y$. As usual, edges of $\AC^s(\mr Y)$ record disjointness. Observe that there is a $1$--Lipschitz map $\AC^s(\mr Y) \to \AC(Y)$, which sends an arc or  curve in $\mr Y$ to its isotopy class in $Y$. 


A \define{witness} for $\AC^s(\mr Y)$ is an essential subsurface $\psub \subset \mr Y$ such that no  representative of a vertex of $\AC^s(\mr Y)$ can be isotoped to be disjoint from $\psub$. The set of (isotopy classes of) witnesses for $\AC^s(\mr Y)$ is denoted $\Omega(\mr Y)$.
\end{definition}

\begin{lemma}\label{Lem:Witness}
Let $Y$ be a $q$--compatible, non-annular subsurface of $S$. 
Let $\psub \subset \mr S$ be a surface with $d_{\psub}(\lambda^{-}, \lambda^{+}) \geq 4$ such that
$\Fill_S(\psub) \subset Y$.
Then $\psub\in \Omega(\mr Y)$ if and only if $\psub \in \mc F_Y$. 
\end{lemma}

\begin{proof}
Suppose $d_{\psub}(\lambda^{-}, \lambda^{+}) \geq 4$ and $\psub$ is isotopic into $Y$. Let $Z = \Fill_S(\psub)$. Then $\psub$ is $q$--compatible and $\int_q \psub \subset \int_q Z \subset \int_q Y$, where the first containment holds by \Cref{Lem:QCompatibleFill} and the second containment holds because $Z \subset Y$. Thus, in particular, $\psub$ is a subsurface of $\mr Y$.

Now, observe that $\psub$ fails to be a witness 
 for $\AC^s(\mr Y)$ if and only if the complement $\mr Y \setminus \psub$ contains an essential arc or curve $\alpha \in \AC^s(\mr Y)$ that is disjoint from $\psub$ and remains essential in $Y$. This happens if and only if $Z = \Fill_S(\psub)$ is a proper subsurface of $Y$, meaning $\psub \notin \mc F_Y$.
\end{proof}

For any curve/arc $\alpha$ in $Y$, we denote by $\mr \alpha$ \emph{any} lifts to a curve/arc of $\mr Y$. That is, $\mr \alpha$ projects to $\alpha$ under the forgetful map $\mr Y \to Y$.

\begin{lemma}\label{Lem:ShufflingLower}
Let  $Y \subset S$ be a non-annular, $q$--compatible subsurface. Given essential arcs or curves $\alpha, \beta \in \AC(Y)$, let $\mr \alpha, \mr \beta$ be arbitrary lifts to $\AC^s(\mr Y)$ obtained by isotoping $\alpha, \beta$ off the singularities in the interior of $\int_q Y$.
Then
\[
d_Y(\alpha, \beta) \prec \sum_{\psub \in \mc F_Y} \left[ d_{\psub} (\mr \alpha,  \mr \beta ) \right].
\]
\end{lemma}

\begin{proof}
First, observe that $d_{Y} (\alpha,  \beta) \leq d_{\AC^s(\mr Y)} (\mr \alpha, \mr \beta)$, because the inclusion-induced map $\AC^s(\mr Y) \to \AC(Y)$ is $1$--Lipschitz. Thus the proof will follow from the following distance formula in $\AC^s(\mr Y)$:
\begin{equation}\label{Eqn:SurvivingEstimate}
d_{\AC^s(\mr Y)} (\mr \alpha, \mr \beta) \asymp \sum_{\psub \in \Omega(\mr Y)} \big[ d_{\psub}(\mr \alpha, \mr \beta) \big] 
 = \sum_{\psub \in \mc F_Y} \big[ d_{\psub}(\mr \alpha, \mr \beta) \big] .
\end{equation}
The second equality of \eqref{Eqn:SurvivingEstimate} is just \Cref{Lem:Witness}. However, the first (coarse) equality is nontrivial. We will derive this estimate from a theorem of Vokes \cite[Corollary 1.2]{Vokes} which, for a family of graphs that will cover our situation, simplifies the axioms for hierarchical hyperbolicity and establishes the desired distance formula.

Vokes works in the general setting of \define{twist-free multicurve graphs}; see \cite[Definition 2.1]{Vokes}. 
Here are the five defining properties of a twist-free multicurve graph $\mc G(\mr Y)$, combined with the reasons why $\AC^s(\mr Y)$ or a sufficiently close variant qualifies:
\begin{enumerate}
\item \emph{The graph $\mc G(\mr Y)$ is connected.} In our setting, the connectedness of $\AC^s(\mr Y)$ follows 
from a lemma of Putman \cite[Lemma 2.1]{Putman}, by a standard argument 
using the change-of-coordinates principle.

\smallskip
\item \emph{Each vertex of $\mc G(\mr Y)$ represents a multicurve in $\mr Y$.} In our setting, this hypothesis is satisfied by the induced subgraph $\C^s(Y)$ whose vertices are surviving curves. Since every arc in $\AC^s(\mr Y)$ is distance $1$ from a curve, this is sufficient for the desired distance formula.

\smallskip
\item \emph{The action of $\Mod(\mr Y)$ on the surface $\mr Y$ induces an isometric action on $\mc G(\mr Y)$.} Kopreski \cite{Kopreski} has loosened this hypothesis to only require an action of the pure mapping class group $\textup{P}\Mod(\mr Y)$. In our setting, the pure mapping class group $\textup{P}\Mod(\mr Y)$ fixes every (singular) point of $Y \setminus \mr Y$, and therefore acts isometrically on $\AC^s(\mr Y)$.

\smallskip
\item \emph{There exists a universal bound on the intersection number $i(a,b)$ for any pair of adjacent vertices $a, b$ of $\mc G(\mr Y)$.} In our setting, adjacent vertices of $\AC^s(\mr Y)$ represent disjoint curves/arcs.

\smallskip
\item \emph{The set of witnesses for $\mc G(\mr Y)$ does not contain annuli.} In our setting, this holds by \Cref{Lem:Witness}, because an annulus cannot fill to the non-annular surface $Y$.
\end{enumerate}

Therefore, Vokes's  \cite[Corollary 1.2]{Vokes} implies the distance formula \eqref{Eqn:SurvivingEstimate}, concluding the proof.
\end{proof}

\begin{remark}
 In the special case where $\mr Y = Y \setminus z$ for a single point $z$, the first coarse equality of \eqref{Eqn:SurvivingEstimate} is a theorem of G\"ultepe, Leininger, and Pho-On \cite[Theorem 5.1]{GultepeLeiningerPhoOn}. This theorem is phrased specifically for the surviving curve graph $\C^s(Y)$ and is quite easy to apply. However, the line of argument in \cite{GultepeLeiningerPhoOn} is particular to the situation where only a single point of $Y$ gets punctured.
\end{remark}

\begin{proof}[Proof of \Cref{Prop:Shuffling}.\eqref{Itm:ProperSubsurfShuffling}]
Let $Y \subset S$ be a proper, non-annular subsurface. As noted at the beginning of this subsection, we may assume that $d_Y(\lambda^+, \lambda^-) \geq 4$, which implies that $Y$ is $q$--compatible.

Choose a nonsingular leaf $\ell^\pm$ of each foliation $\lambda^\pm$. Then $\alpha = \pi_Y(\ell^+)$ and $\beta = \pi_Y(\ell^-)$ are essential multi-arcs in $Y$, which can be obtained as the isotopy classes of all components of intersection between $\ell^\pm$ and the $q$--convex hull $\int_q Y$.
%
For any $q$--compatible surface $\psub \in \mc F_Y$, the projections $\pi_{\psub}(\ell^+)$ and $\pi_{\psub}(\ell^-)$ can likewise be obtained as the isotopy classes of all intersections between $\ell^\pm$ and the $q$--convex hull $\int_q \psub$.

Since the leaves $\ell^\pm$ are already $q$--geodesic, the $q$--projection $\pi_{\psub}^q (\alpha) = \pi_{\psub}^q \circ \pi_Y(\ell^+)$ lands within distance $1$ of $\pi_{\psub}(\ell^+)$, and similarly for $\ell^-$. Thus \Cref{Lem:ShufflingUpper} implies
\[
d_Y(\lambda^+, \lambda^-) \asymp
d_Y(\alpha, \beta)  \succ
 \sum_{\psub \in \mc F_Y} \left[ d_{\psub} \Big( \pi^{q}_{\psub} (\alpha), \pi^{q}_{\psub}(\beta) \Big) \right] \asymp
\sum_{\psub \in \mc F_Y} \big[ d_{\psub}(\lambda^+, \lambda^-) \big].
\]

For the other direction, observe that $\ell^+$ is a $q$--geodesic disjoint from $\sing(q)$, hence every component of $\ell^+ \cap \int_q Y$ is already contained in $\mr Y$. It follows that $\alpha = \pi_Y(\ell^+)$ has a canonical lift $\mr \alpha \in \AC^s(\mr Y)$, such that for every $q$--compatible $\psub$, the projection $\pi_{\psub}(\mr \alpha)$ lies within within distance $1$ of $\pi_{\psub}(\lambda^+)$. An identical statement holds for $\beta$, the canonical lift $\mr \beta$, and $\lambda^-$. Thus \Cref{Lem:ShufflingLower} implies
\[
d_Y(\lambda^+, \lambda^-) \asymp d_Y(\alpha, \beta) \prec \sum_{\psub \in \mc F_Y} \left[ d_{\psub} (\mr \alpha,  \mr \beta ) \right]
\asymp \sum_{\psub \in \mc F_Y} \big[ d_{\psub}(\lambda^+, \lambda^-) \big],
\]
completing the proof.
\end{proof}

\subsubsection{Proving \Cref{Prop:Shuffling}.\eqref{Itm:FullSurfShuffling}}
It remains to consider the situation where $Y = S$ is the whole surface. In this setting, we will once again base our proof on \Cref{Lem:ShufflingUpper,Lem:ShufflingLower}. The challenge is that $d_S(\lambda^+, \lambda^-)$ is infinite, so we must instead replace this distance by the stable translation length $\ell_S(f)$, defined via \eqref{Eqn:StableLength}.

Every section $T$ consists of veering edges that represent vertices of $\AC(\mr S)$. We can also
 use a section $T$ to represent arcs and curves in $S$, as follows.

Let $T$ be any section, and consider $T \cup \sing(q)$ as a graph on $S$. Define $\pi_S(T)$ to be the set of all essential arcs and curves in $S$ whose $q$--geodesic representative is contained in $T \cup \sing(q)$, and traverses no edge of $T$ more than twice. 
Then $\pi_S(T) \neq \emptyset$: indeed, any veering edge $\sigma \subset T$ can be extended by concatenation to a $q$--geodesic arc or curve supported in $T \cup \sing(q)$, which traverses any edge of $T$ at most twice.
Furthermore, $\pi_S(T)$
 is a finite set, such that the intersection number of any pair of elements is uniformly bounded by a bound depending only on $S$.
Now, \eqref{Eqn:StableLength} gives
\begin{equation}\label{Eqn:StableLengthRestate}
\ell_S(f) = \lim_{n \to \infty} \frac{ d_S ( \pi_S(T), \pi_S(f^n(T)) }{n}.
\end{equation}

\begin{lemma}\label{Lem:Shuffling2WholeSurf}
Let $T$ be any section. Then, for any $n \in \NN$,
\[
d_S ( \pi_S(T), \pi_S(f^n(T)) \asymp \sum_{\psub \in \mc F_S} \big[ d_{\psub}(T, f^n(T)) \big],
\]
where the implicit constants are independent of $n$.
\end{lemma}

\begin{proof}
By construction, every curve or arc $\gamma \in \pi_S(T)$ has its $q$--geodesic representative supported in $T$, so $\mathbf{s} \circ c (\gamma) \subset T$. Thus, for every $\psub \in \mc F_S$, \Cref{Def:QProjection} implies that $\pi_{\psub}^q \circ \pi_S(T)$ lies within distance $1$ of $\pi_{\psub}(T)$. Accordingly, applying \Cref{Lem:ShufflingUpper} to $\alpha =  \pi_S(T)$ and 
$\beta =  \pi_S(f^n(T))$
gives
\begin{align*}
d_S ( \pi_S(T), \pi_S(f^n(T))
&\succ \sum_{\psub \in \mc F_S} \left[ d_{\psub} \Big( \pi_{\psub}^q \circ \pi_S(T), \pi_{\psub}^q \circ \pi_S(f^n(T)) \Big) \right] \\
&\asymp
\sum_{\psub \in \mc F_S} \big[ d_{\psub}(T, f^n(T)) \big].
\end{align*}

For the other direction, observe again that any curve or arc $\gamma \in \pi_S(T)$ is a $q$--geodesic consisting of saddle connections in $T$. For any singular point $x \in \gamma \cap \sing(q)$, there are two ways to push $\gamma$ off $x$: one to the left and one to the right (with respect to a transverse orientation of $\gamma$).
Thus $\gamma$ has two canonical lifts to $\AC^s(\mr S)$: one that circumnavigates every singular point on the left, and the other on the right. We define $\mr \gamma$ to be the multi-curve or multi-arc consisting of these two lifts, and observe that $\mr \gamma$ has bounded intersection number with $T$. Then the full lift $\mr \pi_S(T)$ has universally bounded intersection number with $T$, with the bound depending only on $S$. We follow the same process to get a lift $\mr \pi_S(f^n(T))$.
It follows that for every $\psub \in \mc F_Y$, the distance
\[
d_{\psub} (\mr{\pi}_S(T), T) = d_{\psub} \big( \pi_{\psub}(\mr{\pi}_S(T)), \: \pi_{\psub}(T) \big)
\]
is uniformly small, as is $d_{\psub} (\mr{\pi}_S(f^n(T)), f^n(T))$. Applying \Cref{Lem:ShufflingLower} to $\mr \alpha =  \mr \pi_S(T)$ and 
$\mr \beta = \mr \pi_S(f^n(T))$ gives
\[
d_S ( \pi_S(T), \pi_S(f^n(T)) 
\prec \sum_{\psub \in \mc F_Y} \left[ d_{\psub} \big( \mr{\pi}_S(T), \mr{\pi}_S(f^n(T)) \big) \right] \\
\asymp \sum_{\psub \in \mc F_S} \big[ d_{\psub}(T, f^n(T)) \big],
\]
completing the proof.
\end{proof}

To complete the proof of \Cref{Prop:Shuffling}.\eqref{Itm:FullSurfShuffling}, we need to partition the sum $ \sum_{\psub \in \mc F_S} \big[ d_{\psub}(T, f^n(T)) \big]$ into sums over $f$--orbits, and then estimate how many surfaces in the $f$--orbit of $\psub$ survive the cutoff $[ \cdot ]_K$. This is accomplished in the following lemma.

\begin{lemma}\label{Lem:PuncturedSum}
Let $T$ be an $f$--section, and let $\psub \in \mc F_S$ be a proper subsurface of $\mr S$. Then, for all $n \in \NN$ and all $K \geq 25$, we have
\[
n \cdot \big[ d_{\psub}(\lambda^+, \lambda^-) \big]_K \asymp \sum_{i \in \ZZ} 
\big[ d_{f^i(\psub)}(T, f^n(T)) \big]_{K},
\]
where the implicit constants are independent of $\psub$ and $n$.
\end{lemma}

\begin{proof}
We may suppose that $d_{\psub}(\lambda^+, \lambda^-) > 12$, as otherwise both sides of the claimed coarse equality are zero by \Cref{Thm:NicePath}. Since $\Fill_S(\psub) = S$,  \Cref{Lem:FillingOverlap} implies  $\psub$ is overlapped by every nonzero power of $f$. Thus \Cref{Lem:fSectionSum} applies to $\psub$. Thus we may compute as follows:
\begin{align*}
\sum_{i \in \ZZ} 
\big[ d_{f^i(\psub)}(T, f^n(T)) \big]_{K} 
&= \sum_{j = 1}^{n} \sum_{i \in \ZZ} 
\big[ d_{f^{ni} (f^j (\psub))}(T, f^{n}(T)) \big]_{K} \\
&\asymp \sum_{j = 1}^{n} \big[ d_{f^j(\psub)}(\lambda^+, \lambda^-) \big]_K \\
& = n \cdot \big[ d_{\psub}(\lambda^+, \lambda^-) \big]_K,
\end{align*}
where the one coarse equality is by \Cref{Lem:fSectionSum}  with $f^{n}$ in place of $f$.
\end{proof}

\begin{proof}[Proof of \Cref{Prop:Shuffling}.\eqref{Itm:FullSurfShuffling}]
Let $T$ be an $f$--section, and let $\pi_S$ be the projection from sections to $\AC(S)$ defined above. We can now compute as follows:
\begin{align*}
\ell_S(f) &= \lim_{n \to \infty} \frac{ d_S ( \pi_S(T), \pi_S(f^n(T)) }{n} \\
&\asymp \lim_{n \to \infty}  \sum_{\psub \in \mc F_S} \frac{1}{n} \big[ d_{\psub}(T, f^n(T)) \big]  \\
& = \lim_{n \to \infty} \frac{1}{n} \big[ d_{\mr S}(T, f^n(T)) \big]  + \lim_{n \to \infty} \sum_{\psub \in \mc F_S \setminus \{ \mr S \}} \frac{1}{n} \big[ d_{\psub}(T, f^n(T)) \big] \\
& \asymp  \lim_{n \to \infty} \frac{1}{n} \big( n \cdot \ell_{\mr S}(\mr f) \big) \quad + \quad
 \lim_{n \to \infty}  \sum_{\langle f \rangle \psub \subset \mc F_S \setminus \{ \mr S \}} \frac{1}{n}
 \Big(  n \cdot \big[ d_{\psub}(\lambda^+, \lambda^-) \big] \Big) \\
 & = \qquad \qquad \ell_{\mr S}(\mr f) \qquad \qquad + \qquad \sum_{\langle f \rangle \psub \subset \mc F_S \setminus \{\mr S\}} \big[ d_{\psub}(\lambda^+,\lambda^-) \big].
\end{align*}
Here, the first equality is \eqref{Eqn:StableLengthRestate}. The next coarse equality is \Cref{Lem:Shuffling2WholeSurf}. The next equality holds by separating out the summand for the whole punctured surface $\mr S$ from the rest of the sum. The next coarse equality holds by applying \Cref{Lem:SectionsAlongGeodesic} to distances in $\mr S$ and \Cref{Lem:PuncturedSum} to every $f$--orbit of proper subsurfaces $\psub$. 
(Here, it is crucial that the coarse constants in  \Cref{Lem:PuncturedSum} are independent of $\psub$ and $n$, enabling us to apply \Cref{Lem:HierarchyLikeSum} to the whole sum over orbits.)
The final equality holds by cancellation and taking limits.
\end{proof}

\section{Examples and constructions} \label{Sec:Examples}\label{sec:FPF}

In this section, we construct various examples that, in spirit, demonstrate the sharpness of our results and the necessity of our hypotheses. 


In the following construction, we describe a family of pseudo-Anosovs whose fixed points do not grow, but whose Teichm{\"u}ller translation lengths become arbitrarily large. Of course, these maps are not strongly irreducible.  

\begin{construction}[Twisting along nonoverlapped annuli] \label{Const:twist}\label{twist}
Let $f \colon S \to S$ be a pseudo-Anosov homeomorphism and let $\alpha$ be a curve in $S$ such that $f(\alpha)$ is disjoint from $\alpha$. Then there is an annulus $A$ about $\alpha$ 
that is disjoint from $f(A)$; in particular, $A$ contains no fixed points of $f$. Therefore, if we denote by $t_\alpha$ a Dehn twist homeomorphism supported on $A$, the map $f \circ t_\alpha^n$ has no more fixed points than $f$ for all $n \in \mathbb{Z}$. Moreover, for all but finitely many $n$, these maps are isotopic to pseudo-Anosov homeomorphisms $f_n$ with the property that $\# \Fix(f_n) \leq \#\Fix(f \circ t_\alpha^n) \leq \#\Fix(f)$. We note that the subsurface distances in the annular complex  $\AC(\alpha)$ between the invariant laminations of $f_n$ go to infinity as $n \to \infty$, hence $\ell_{\T}(f_n) \to \infty$. 
\end{construction}

One can make the discrepancy between fixed points and translation length even more dramatic, as in the following example of pseudo-Anosovs with \textit{no} fixed points:

\begin{construction}[Fibered cones]\label{Const:cone}
Let $M$ be a fibered hyperbolic $3$--manifold such that $b_1(M) \geq 2$. Let $S_1$ and $S_2$ be distinct fibers in the same fibered cone of $H^1(M)$, with associated pseudo-Anosov monodromies $f_1$ and $f_2$. Let $\Phi$ be the pseudo-Anosov flow on $M$ defined by this fibered cone. Then, after an isotopy, $S_1$ and $S_2$ are cross-sections of $\Phi$ (i.e. they transversely meet every orbit) and the first return maps to these surfaces are $f_1$ and $f_2$, respectively.
Observe that any closed orbit of $\Phi$ intersects a fiber $S_i$ at least once, and that fixed points of $f_i$ correspond to closed orbits that intersect $S_i$ exactly once.

Now, let $S_3$ be the cross-section obtained from $S_1 , S_2$ by smoothing their intersections while maintaining transversality to $\Phi$. Thus $S_3$ represents the homology class $[S_1] + [S_2]$. 
Then any closed orbit of $\Phi$ intersects $S_3$ at least twice, hence the first return map $f_3 \from S_3 \to S_3$ is a pseudo-Anosov without fixed points. Note that to make $S_3$ connected, we must choose $S_1$ and $S_2$ so that  $[S_1] + [S_2]$ is not a multiple. 

Moreover, if $\alpha$ is a component of the intersection $S_1\cap S_2$ that is essential in each $S_i$, then the smoothing $S_3$ contains two disjoint preimages of $\alpha$ -- say, $\alpha'$ and $\alpha''$ -- such that the first return map $f_3$ sends $\alpha'$ to $\alpha''$. We conclude that 
 $f_{3}$ is both fixed-point-free and not strongly irreducible. 
%
\end{construction}

\begin{proof}[Proof of \Cref{prop:FPF}]
By plugging the fixed-point-free map $f_3 \from S_3 \to S_3$ of \Cref{Const:cone} into the twisting procedure of \Cref{Const:twist}, we obtain infinitely many conjugacy classes of fixed-point-free pseudo-Anosovs on the same surface $S_3$, with dilatations going to $\infty$. 
\end{proof}


In \Cref{twist}, the growing Teichm{\"u}ller translation length comes from iterating a Dehn twist, which does not arbitrarily increase the volume of the mapping torus. However, it is easy to generalize the construction to produce pseudo-Anosov homeomorphisms whose number of fixed points stays bounded but whose mapping tori volume is unbounded: 

\begin{example}[Unbounded volume] \label{ex:vol}
Let $f \colon S \to S$ be a pseudo-Anosov map such that $Y \subset S$ is a non-annular, non-overlapped essential subsurface. That is, $f(Y) \cap Y = \emptyset$. In particular, since $Y$ is non-annular, the interior of $Y$ itself supports a pseudo-Anosov homeomorphism $g$. Then just as above, for all but finitely many $n$, the map $f \circ g^n$ is isotopic to a pseudo-Anosov homeomorphism whose number of fixed points is no greater than that of $f$. Moreover, it follows from
the work of Brock \cite{Brock:fibered} that the volume of the mapping tori of $f \circ g^n$ go to $\infty$.
\end{example}

\bibliographystyle{amsplain}
\bibliography{biblio.bib}

\end{document}